\documentclass[reqno,10pt]{amsart}
\usepackage{amscd,amssymb,verbatim,array}
\usepackage{hyperref}
\usepackage{amsmath}
\usepackage[active]{srcltx} % SRC Specials: DVI [Inverse] Search

\numberwithin{equation}{section} \theoremstyle{plain}

\newtheorem{theorem}{Theorem}[section]
\newtheorem{lemma}[theorem]{Lemma}

\newtheorem{corollary}[theorem]{Corollary}

\theoremstyle{definition}

\theoremstyle{remark}

\newcommand{\Det}{\operatorname{Det}}

\newcommand{\Dim}{\operatorname{dim}}

\newcommand{\rank}{\operatorname{rank}}

\newcommand{\Ker}{\operatorname{ker}}

\newcommand{\Tr}{\operatorname{Tr}}
\newcommand{\re}{\operatorname{Re}}

\newcommand{\Aa}{\operatorname{I}}
\newcommand{\Bb}{\operatorname{II}}
\newcommand{\Cc}{\operatorname{III}}
\newcommand{\Dd}{\operatorname{IV}}
\newcommand{\Ee}{\operatorname{V}}

\newcommand{\ddet}{\operatorname{det}}
\newcommand{\Id}{\operatorname{Id}}

\newcommand{\vol}{\operatorname{vol}}

\begin{document}

\title[BFK-gluing formula and the curvature tensors]
{The BFK-gluing formula and the curvature tensors on a $2$-dimensional compact hypersurface}
\author{Klaus Kirsten}

\address{GCAP-CASPER, Department of Mathematics, Baylor University, Waco, TX 76796, USA}

\email{Klaus\_Kirsten@Baylor.edu}

\author{Yoonweon Lee}

\address{Department of Mathematics Education, Inha University, Incheon, 402-751, Korea}

\email{yoonweon@inha.ac.kr}

\thanks{The first author was supported by the Baylor University Summer Sabbatical and Research Leave Program.
The second author was supported by the National Research Foundation of Korea with the Grant number 2016R1D1A1B01008091}
%\date{\today}

\begin{abstract}
In the proof of the BFK-gluing formula for zeta-determinants of Laplacians there appears a real polynomial
whose constant term is an important ingredient in the gluing formula.
This polynomial is determined by geometric data on an arbitrarily small collar neighborhood of a cutting hypersurface.
In this paper we express the coefficients of this polynomial in terms of the scalar and principal curvatures of the cutting hypersurface embedded in the manifold when this hypersurface is $2$-dimensional.
Similarly, we express some coefficients of the heat trace asymptotics of the Dirichlet-to-Neumann operator in terms of the scalar and principal curvatures of the cutting hypersurface.
\end{abstract}

\maketitle

\section{Introduction}

\setcounter{equation}{0}

\vspace{0.2 cm}

The zeta-determinant of a Laplacian on a compact oriented Riemannian manifold is a global spectral invariant which plays an important role in geometry, for example in the context of analytic torsion \cite{ray71-7-145},
topology and mathematical physics \cite{eliz95b,Ki1}. However, it is almost impossible to compute the zeta-determinant precisely except for
a few cases and hence the gluing formula of this invariant is meaningful at least for some computational reason.
The gluing formula for the zeta-determinants of Laplacians on a compact oriented Riemannian manifold was proved by Burghelea, Friedlander and Kappeler in \cite{BFK} by using the Dirichlet boundary condition and Dirichlet-to-Neumann operator.
We will call this formula the BFK-gluing formula (see also \cite{Ca,Fo}).
In the proof of this gluing formula there appears a polynomial of degree less than half of the dimension of the underlying manifold \cite{Ca},
whose constant term is an important ingredient of the BFK-gluing formula.
This polynomial is determined by an arbitrarily small collar neighborhood of a cutting hypersurface and it vanishes when the cutting hypersurface is odd dimensional
\cite{KL3}. Moreover in \cite{KL3}, when a collar neighborhood of a cutting hypersurface is isometric to a warped product manifold, this polynomial is computed  precisely in terms of a warping function. The coefficients of this polynomial are closely related to the small time asymptotic expansion of the heat trace of Laplacians and the asymptotic expansion
of $\ln \Det R(\lambda)$ for $\lambda \rightarrow \infty$, where $R(\lambda)$ is a one parameter family of the Dirichlet-to-Neumann operator defined in (\ref{E:1.2}) below.

Motivated by the work of \cite{Gi2}, \cite{Ki1} and \cite{PS}, the authors recently recognized that the coefficients of the small time asymptotic expansion of the heat trace of $R(\lambda)$ and the asymptotic expansion of $\ln \Det R(\lambda)$ for $\lambda \rightarrow \infty$
can be expressed by the scalar curvatures and principal curvatures of a cutting hypersurface in an underlying manifold.
In this paper we are going to use the arguments presented in \cite{PS} to
express the coefficients of these asymptotic expansions by using the curvature tensors of the cutting hypersurface. These results together with those of \cite{Gi2} or \cite{Ki1} enable us to express the coefficients of the polynomial
by those curvature tensors.

In principle, the method presented in this paper works for all dimensions. However, given the complications when
computing the homogeneous symbol of $R(\lambda)$ defined on the cutting hypersurface, when determining the coefficients
of the polynomial we are going to restrict our argument
to the case that the cutting hypersurface is a $2$-dimensional compact Riemannian manifold.

Let $(M, g)$ be an $m$-dimensional compact oriented Riemannian manifold with boundary $\partial M$,
where $\partial M$ may be empty.
We choose a closed hypersurface ${\mathcal N}$ of $M$ such that ${\mathcal N} \cap \partial M = \emptyset$.
We denote by $M_{0}$ the closure of $M - {\mathcal N}$ and by $g_{0}$ the metric on $M_{0}$ induced from $g$ so that
$(M_{0}, g_{0})$ is a compact Riemannian manifold with boundary $\partial M \cup {\mathcal N}^{+} \cup {\mathcal N}^{-}$,
where ${\mathcal N}^{+} = {\mathcal N}^{-} = {\mathcal N}$.
We here note that $M_{0}$ may or may not be connected.
We denote by ${\mathcal E} \rightarrow M$ a vector bundle of rank $r_{0}$ and denote by ${\mathcal E}_{0} \rightarrow M_{0}$
the pull back bundle of ${\mathcal E}$ by the identification map $p : M_{0} \rightarrow M$.
We choose an inner product $(~, ~)$ on ${\mathcal E}$, which
together with the Riemannian metric $g$ gives an inner product $\langle ~ , ~ \rangle$ on $L^{2}({\mathcal E})$.
We also denote by $\Delta_{M}$ a Laplacian on $M$ acting on smooth sections of ${\mathcal E}$
with the Dirichlet boundary condition on $\partial M$ in case of $\partial M \neq \emptyset$.
By a Laplacian we mean a non-negative symmetric differential operator of order $2$ with principal symbol
$\sigma_{L}(\Delta_{M})(x, x_{m}, \xi, \xi_{m}) = \parallel \xi \parallel^{2} + |\xi_{m}|^{2}$, where $(x, x_{m})$ is a coordinate on $M$ and
$(\xi, \xi_{m})$ is a coordinate on $T_{(x, x_{m})}^{\ast}M$.
We also denote by $\Delta_{M_{0}, D}$ the Laplacian on $M_{0}$ induced from $\Delta_{M}$ with the Dirichlet boundary condition on
$\partial M \cup {\mathcal N}^{+} \cup {\mathcal N}^{-}$.
We choose a unit normal vector field $\partial_{x_{m}}$ to ${\mathcal N}$ which points outward on ${\mathcal N}^{+}$ and
inward on ${\mathcal N}^{-}$.
For $0 \leq \lambda \in {\mathbb R}$, we define the Dirichlet-to-Neumann operator
$R(\lambda) : C^{\infty}({\mathcal E}|_{{\mathcal N}}) \rightarrow C^{\infty}({\mathcal E}|_{{\mathcal N}})$ as follows.
For $f \in C^{\infty}({\mathcal E}|_{{\mathcal N}})$, choose $\psi \in C^{\infty}({\mathcal E}_{0})$ satisfying
\begin{eqnarray}   \label{E:1.1}
(\Delta_{M_{0}} +\lambda) \psi = 0, \qquad \psi|_{{\mathcal N}^{+}} = \psi|_{{\mathcal N}^{-}} = f, \qquad \psi|_{\partial M} = 0.
\end{eqnarray}
\noindent
In fact, $\psi$ is given by $\psi = {\widetilde f} - (\Delta_{M_{0}, D} + \lambda)^{-1} (\Delta_{M_{0}} + \lambda) {\widetilde f}$,
where ${\widetilde f}$ is an arbitrary extension of $f$ to $M_{0}$.
We define
\begin{eqnarray}   \label{E:1.2}
R(\lambda) f & = & \left( \nabla_{\partial_{x_{m}}} \psi \right)\big|_{{\mathcal N}^{+}} -
\left( \nabla_{\partial_{x_{m}}} \psi \right)\big|_{{\mathcal N}^{-}}.
\end{eqnarray}
\noindent
It is well known that $R(\lambda)$ is a non-negative self-adjoint elliptic operator of order $1$.

For $\nu = [\frac{m}{2}] + 1$, it was shown in \cite{BFK} (see also \cite{Ca}) that
\begin{eqnarray}    \label{E:1.3}
\frac{d^{\nu}}{d \lambda^{\nu}} \left\{ \ln \Det \left( \Delta_{M} + \lambda \right) -
\ln \Det \left( \Delta_{M_{0}, D} + \lambda \right) \right\} & = &  \frac{d^{\nu}}{d \lambda^{\nu}} \ln \Det R(\lambda),
\end{eqnarray}
\noindent
which implies that there exists a real polynomial $P(\lambda) = \sum_{j=0}^{\nu-1} a_{j} \lambda^{j}$ such that
\begin{eqnarray}   \label{E:1.4}
\ln \Det \left( \Delta_{M} + \lambda \right) -
\ln \Det \left( \Delta_{M_{0}, D} + \lambda \right) & = & \sum_{j=0}^{\nu-1} a_{j} \lambda^{j} +  \ln \Det R(\lambda).
\end{eqnarray}
\noindent
{\it Remark} : If $M_{0}$ has two components, say, $M_{1}$ and $M_{2}$, it is understood that
\begin{eqnarray*}
\ln \Det \left( \Delta_{M_{0}, D} + \lambda \right) & = &
\ln \Det \left( \Delta_{M_{1}, D} + \lambda \right) + \ln \Det \left( \Delta_{M_{2}, D} + \lambda \right).
\end{eqnarray*}
\noindent
It is known that $P(\lambda)$ is determined by some geometric data on an arbitrarily small collar neighborhood $U$ of
${\mathcal N}$, and that $P(\lambda) = 0$ when $\Dim M$ is even \cite{KL3}.
For a smooth function $f : [a, ~ b] \rightarrow {\mathbb R}^{+}$, we consider a warped product manifold $[a, ~ b] \times_{f} {\mathcal N}$.
When $U$ is isometric to $[a, ~ b] \times_{f} {\mathcal N}$ and ${\mathcal E}$ is a trivial line bundle,
$P(\lambda)$ was computed explicitly in terms of a warping function $f$ in \cite{KL1}.

Suppose that for $t \rightarrow 0^{+}$,
\begin{eqnarray}  \label{E:1.5}
\Tr \left( e^{- t \Delta_{M}} - e^{- t \Delta_{M_{0}, D}} \right) \sim \sum_{j=0}^{\infty} c_{j} t^{\frac{-m + j}{2}},
\end{eqnarray}
where $c_{0} = 0$.
As $\lambda \rightarrow \infty$, the left hand side of (\ref{E:1.4}) has the following asymptotic expansion
(Lemma 2.1 in \cite{KL2},  eq. (5.1) in \cite{Vo}):
\begin{eqnarray}    \label{E:1.6}
& & \ln \Det \left( \Delta_{M} + \lambda \right) - \ln \Det \left( \Delta_{M_{0}, D} + \lambda \right) ~ \sim ~
- \sum_{\stackrel{j=0}{j\neq m}}^{N} c_{j} \frac{d}{ds}\left. \left( \frac{\Gamma(s - \frac{m-j}{2})}{\Gamma(s)} \right)\right|_{s=0} \lambda^{\frac{m-j}{2}}   \\
& & \hspace{1.5 cm}  + ~ c_{m} \ln \lambda ~ + ~ \sum_{j=0}^{m-1} c_{j}
\left. \left( \frac{\Gamma(s - \frac{m-j}{2})}{\Gamma(s)} \right)\right|_{s=0} \lambda^{\frac{m-j}{2}} \ln \lambda ~ + ~ O(\lambda^{-\frac{N + 1 - m}{2}}),  \nonumber
\end{eqnarray}
\noindent
 where $N \geq m+1$.
We note that the constant term does not appear in the above expansion.
It was shown in the Appendix of \cite{BFK} that $R(\lambda)$ is an elliptic $\Psi$DO of order 1 with parameter of weight $2$ and that
for $\lambda \rightarrow \infty$,
$\ln \Det R(\lambda)$ has the following asymptotic expansion,
\begin{eqnarray}   \label{E:1.7}
\ln \Det R (\lambda) & \sim & \sum_{j=0}^{\infty} \pi_{j} \lambda^{\frac{m-1-j}{2}} + \sum_{j=0}^{m-1} q_{j} \lambda^{\frac{m-1-j}{2}} \ln \lambda,
\end{eqnarray}
\noindent
where each $\pi_{i}$ and $q_{j}$ can be computed by an integral of some density on ${\mathcal N}$ consisting of the homogeneous symbol of the resolvent of $R(\lambda)$.
The comparison of these two asymptotic expansions with (\ref{E:1.4}) shows that
\begin{eqnarray}    \label{E:1.8}
a_{0} & = & - \pi_{m-1}.
\end{eqnarray}
Let $\{ \phi_{1}, \cdots, \phi_{l} \}$ be an orthonormal basis of $\Ker \Delta_{M}$, where $l = \Dim \Ker \Delta_{M}$.
We define a positive definite symmetric $l \times l$ matrix $A_{0}$ as follows,
\begin{eqnarray}    \label{E:1.9}
d_{ij} & = & \langle \phi_{i}|_{{\mathcal N}}, ~ \phi_{j}|_{{\mathcal N}} \rangle_{{\mathcal N}}, \qquad
A_{0} ~ = ~ \left( d_{ij} \right).
\end{eqnarray}
\noindent
Then, it is known \cite{Le1,MM} that
\begin{eqnarray}   \label{E:1.10}
& & \lim_{\lambda \rightarrow 0^{+}} \ln \Det \left( \Delta_{M_{0}, D} + \lambda \right) ~ = ~ \ln \Det \Delta_{M_{0}, D},  \\
& & \lim_{\lambda \rightarrow 0^{+}} \left( \ln \Det \left( \Delta_{M} + \lambda \right)  - \ln \Det R(\lambda) \right) ~ = ~
\ln \ddet A_{0} + \ln \Det \Delta_{M} - \ln \Det R,   \nonumber
\end{eqnarray}
\noindent
where $R := R(0)$.
This fact leads to the final form of the BFK-gluing formula as follows:
\begin{eqnarray}    \label{E:1.11}
\ln \Det \Delta_{M} - \ln \Det \Delta_{M_{0}, D}  & = & - \pi_{m-1} + \ln \ddet A_{0} + \ln \Det R.
\end{eqnarray}
In this paper we are going to compute the homogeneous symbol of $R(\lambda)$ using the boundary normal coordinate system
on a collar neighborhood of ${\mathcal N}$.
Using this result together with the method presented in \cite{PS}, we are going to express some coefficients of the small time
asymptotic expansion of $\Tr e^{- t R(\lambda)}$ and the asymptotic expansion of $\ln \Det R(\lambda)$ for
$\lambda \rightarrow \infty$ by the scalar curvatures and principal curvatures of ${\mathcal N}$ in $M$.
These results enable us to express the coefficients of $P(\lambda)$ by the scalar and principal curvatures on ${\mathcal N}$.
As a byproduct we can describe the value of $\zeta_{R(\lambda)}(s)$ at $s=0$ by using some coefficient in (\ref{E:1.5}).
In fact, to compute the homogeneous symbol of $R(\lambda)$, we need to perform a long tedious computation.
To avoid further complications, at some point we are going to restrict the argument
to the case that ${\mathcal N}$ is a $2$-dimensional manifold, even though the method of this paper works for any dimension.
\section{The homogeneous symbol of the Dirichlet-to-Neumann operator $R(\lambda)$}
In this section we are going to discuss the homogeneous symbol of the one parameter family $R(\lambda)$ of the Dirichlet-to-Neumann operators using
the so-called boundary normal coordinate system on a collar neighborhood of ${\mathcal N}$.
To compute the asymptotic expansion of $\ln \Det R(\lambda)$, we regard $\lambda$ as a parameter of weight $2$ so that
the principal symbol of $\Delta_{{\mathcal N}} + \lambda$ is
$\sigma_{L}(\Delta_{{\mathcal N}}) + \lambda$. In this case, $\Delta_{{\mathcal N}} + \lambda$ and $R(\lambda)$ are called the elliptic operators with parameter of weight $2$.
For details of elliptic operators with parameter, we refer to the Appendix of \cite{BFK} or \cite{GS} and \cite{Sh}.

Recall that ${\mathcal E} \rightarrow M$ is a vector bundle of rank $r_{0}$ and $\Delta_{M}$ is a Laplacian acting on smooth sections of ${\mathcal E}$.
It is well known \cite{BGV, Gi2} that there exists a connection
$\nabla : C^{\infty}({\mathcal E}) \rightarrow C^{\infty}(T^{\ast}M \otimes {\mathcal E})$ and
a bundle endomorphism $E : {\mathcal E} \rightarrow {\mathcal E}$ such that
\begin{eqnarray}   \label{E:2.1}
\Delta_{M} & = & - \left( \Tr \nabla^{2} + E \right).
\end{eqnarray}
Next, we are going to describe $\Delta_{M}$ by using local coordinates and a local frame for ${\mathcal E}$.
We first choose a normal coordinate $x = (x_{1}, \cdots, x_{m-1})$ on an open neighborhood of $p \in {\mathcal N}$ with $p = (0, \cdots, 0)$.
Let $u$ be the parameter along the unit speed geodesic $\gamma_{x}(u)$ starting from $x$ with the direction of $\partial_{x_{m}}$ on a collar neighborhood of ${\mathcal N}$,
which is outward on ${\mathcal N}^{+}$ and inward on ${\mathcal N}^{-}$.
Then, $( x,u) = (x_{1}, \cdots, x_{m-1}, u)$ gives a local coordinate system, which is called the boundary normal coordinate system. We will use $u=x_m$ when
notationally convenient, in particular in summations considered later.
Since ${\mathcal N}$ is compact, we can choose a uniform constant $\epsilon_{0} > 0$ such that $\gamma_{x}(u)$ is well defined
for any $x \in {\mathcal N}$ and  $- \epsilon_{0} \leq u \leq \epsilon_{0}$. Then,
\begin{eqnarray}      \label{E:2.2}
U_{\epsilon_{0}} & := & \{ (x, x_{m}) \mid  - \epsilon_{0} < x_{m} < \epsilon_{0}, ~  x \in {\mathcal N} ~ \}
\end{eqnarray}
\noindent
is a collar neighborhood of
${\mathcal N}$ in $M$ with ${\mathcal N}$ being exactly the level of $x_{m} = 0$.
We denote
\begin{eqnarray}    \label{E:2.3}
{\mathcal N}_{x_{m}} & := & \{ (x, x_{m}) \mid ~  x \in {\mathcal N} ~ \} ,
\end{eqnarray}
\noindent
which is diffeomorphic to ${\mathcal N}$.
In this coordinate system the metric tensor is given as follows:
\begin{eqnarray}   \label{E:2.4}
g & = & \sum_{\alpha,\beta = 1}^{m-1} g_{\alpha\beta}(x, x_{m}) dx^{\alpha} dx^{\beta} + (dx^{m})^{2},
\end{eqnarray}
\noindent
where $g_{\alpha\beta}(x, 0) = \delta_{\alpha\beta} + O(x^{2})$.
%Throughout this paper we sometimes use Einstein's summation convention, by which equal lower and upper
%indices are summed over.
In the following, the Greek indices $\alpha, \beta, \cdots$ are always assumed to run from $1$ to $m-1$
and the Roman indices $i, j, k, \cdots$ are assumed to run from $1$ to $m$.
We denote
\begin{eqnarray}   \label{E:2.5}
\left( g_{\alpha\beta}(x, x_{m}) \right)^{-1} & := & \left( g^{\alpha\beta}(x, x_{m}) \right), \qquad
|g|(x, x_{m}) ~ := ~ \ddet \left( g_{\alpha\beta}(x, x_{m}) \right).
\end{eqnarray}
\noindent
We also denote
\begin{eqnarray}    \label{E:2.6}
\frac{\partial}{\partial x_{k}} g_{ij} = g_{ij,k}, \qquad  \frac{\partial^{2}}{\partial x_{k}^{2}} g_{ij} = g_{ij,kk},   \qquad
\frac{\partial}{\partial x_{k}} g^{ij} = g^{ij,k}, \qquad     \frac{\partial^{2}}{\partial x_{k}^{2}} g^{ij} = g^{ij,kk}.
\end{eqnarray}
\noindent
On $U_{\epsilon_0}$ we choose a local frame $\{ \psi_{1}, \cdots, \psi_{r_{0}} \}$ and fix it throughout this paper, furthermore we denote by
$\omega : C^{\infty} \left( (T^{\ast}M \otimes {\mathcal E})|_{U_{\epsilon_{0}}} \right) \rightarrow C^{\infty}({\mathcal E}|_{U_{\epsilon_{0}}})$ the connection form of $\nabla$ on $U_{\epsilon_0}$
with respect to $\{ \psi_{1}, \cdots, \psi_{r_{0}} \}$, {\it i.e.}
\begin{eqnarray}  \label{E:2.7}
\nabla_{\partial_{x_{j}}} \psi_{k} & = & \omega_{j} \psi_{k}, \qquad \text{where} \qquad  \omega_{j} := \omega(\partial_{x_{j}}).
\end{eqnarray}
\noindent
In the local coordinate $(x, x_{m})$ on $U_{\epsilon_{0}}$, $(\ref{E:2.1})$ is expressed as follows \cite{Gi2}.
For $\phi \in C^{\infty}({\mathcal E})$, we have
\begin{eqnarray}    \label{E:2.8}
\Delta_{M} \phi & = & - \sum_{i,j=1}^m \, g^{ij} \left( \nabla_{\partial_{x_{i}}}  \nabla_{\partial_{x_{j}}} \phi -
\nabla_{\nabla_{\partial_{x_{i}}} \partial_{x_{j}}} \phi \right) - E \phi  \\
& = & - \sum_{i,j=1}^m \, g^{ij} \left\{ (\partial_{x_{i}} + \omega_{i}) (\partial_{x_{j}} + \omega_{j}) -
\sum_{l=1}^m \, \Gamma_{ij}^{l} (\partial_{x_{l}} + \omega_{l})  \right\} \phi - E \phi     \nonumber  \\
& = & \sum_{i,j=1}^m \left\{ - g^{ij} \partial_{x_{i}} \partial_{x_{j}}  + \sum_{l=1}^m\, g^{ij} \Gamma_{ij}^{l} \partial_{x_{l}}
- 2 g^{ij} \omega_{j} \partial_{x_{i}}
- g^{ij} \left( \partial_{x_{i}} \omega_{j} + \omega_{i} \omega_{j} - \sum_{l=1}^m \Gamma_{ij}^{l} \omega_{l} \right)\right\} \phi  - E  \phi . \nonumber
\end{eqnarray}
\begin{lemma}  \label{Lemma:2.1}
In the boundary normal coordinate system given in (\ref{E:2.4}), we have
\begin{eqnarray*}
& (1) &\sum_{i,j=1}^m\left(  - g^{ij} \partial_{x_{i}} \partial_{x_{j}}  + \sum_{l=1}^m \, g^{ij} \Gamma_{ij}^{l} \partial_{x_{l}} - 2 g^{ij} \omega_{j} \partial_{x_{i}} \right)  ~ = ~
- \partial_{x_{m}}^{2} - \left( \frac{1}{2} \sum_{\alpha , \beta =1}^{m-1}\, g^{\alpha\beta} g_{\alpha\beta,m} ~ + 2 \omega_{m} \right) \partial_{x_{m}} \\
&  &\hspace{2.0 cm}  - ~\sum_{\alpha , \beta =1}^{m-1} \left( g^{\alpha\beta} \partial_{x_{\alpha}} \partial_{x_{\beta}} +
 \left( \frac{1}{2} g^{\alpha\beta} (\partial_{x_{\alpha}} \ln |g|) +
(\partial_{x_{\alpha}} g^{\alpha\beta}) \right) \partial_{x_{\beta}} + 2 g^{\alpha\beta} \omega_{\beta} \partial_{x_{\alpha}} \right).  \\
& (2) & - \sum_{i,j=1}^m \, g^{ij} \left( \partial_{x_{i}} \omega_{j} + \omega_{i} \omega_{j} - \sum_{l=1}^m\, \Gamma_{ij}^{l} \omega_{l} \right)  \\
& = &  - \left( \partial_{x_{m}} \omega_{m} + \omega_{m} \omega_{m} + \frac{1}{2} \sum_{\alpha , \beta =1}^{m-1} \, g^{\alpha\beta} g_{\alpha\beta,m} \omega_{m} \right)
-\sum_{\alpha , \beta =1}^{m-1}\, g^{\alpha\beta} \left( \partial_{x_{\alpha}} \omega_{\beta} + \omega_{\alpha} \omega_{\beta} -
\sum_{\gamma =1}^{m-1}\, \Gamma_{\alpha\beta}^{\gamma} \omega_{\gamma} \right).
\end{eqnarray*}
\end{lemma}
\begin{proof}
Using the boundary normal coordinate system, we get
\begin{eqnarray*}
- \sum_{i,j=1}^m\, g^{ij} \partial_{x_{i}} \partial_{x_{j}} & = & - \partial_{x_{m}}^{2} - \sum_{\alpha , \beta =1}^{m-1} \, g^{\alpha\beta} \partial_{x_{\alpha}} \partial_{x_{\beta}}, \qquad
- 2 \sum_{i,j=1}^m \, g^{ij} \omega_{j} \partial_{x_{i}} ~ = ~ - 2 \omega_{m} \partial_{x_{m}} - 2 \sum_{\alpha , \beta =1}^{m-1}\,  g^{\alpha\beta} \omega_{\beta} \partial_{x_{\alpha}}, \\
\sum_{i,j,l=1}^m \, g^{ij} \Gamma_{ij}^{l} \partial_{x_{l}} & = & \sum_{i,j=1}^m \, g^{ij} \Gamma_{ij}^{m} \partial_{x_{m}}  +  \sum_{i,j=1}^m \sum_{\alpha =1}^{m-1} g^{ij} \Gamma_{ij}^{\alpha} \partial_{x_{\alpha}}.
\end{eqnarray*}
Since
\begin{eqnarray}     \label{E:2.33}
\Gamma_{ij}^{k} & = & \frac{1}{2} \sum_{s=1}^{m} g^{ks} \left( g_{js,i} + g_{si,j} - g_{ij,s} \right),
\end{eqnarray}
we have
\begin{eqnarray*}
\Gamma^{m}_{ij} & = & - \frac{1}{2} g_{ij,m},
\end{eqnarray*}
which leads to
\begin{eqnarray}   \label{E:2.9}
\sum_{i,j=1}^m \,g^{ij} \Gamma^{m}_{ij} & = & - \frac{1}{2}\sum_{i,j=1}^m \,  g^{ij} g_{ij,m} ~ = ~ - \frac{1}{2} \sum_{\alpha , \beta =1}^{m-1} \, g^{\alpha\beta} g_{\alpha\beta,m}.
\end{eqnarray}
Hence, we get
\begin{eqnarray*}
\sum_{i,j=1}^m\, g^{ij} \Gamma^{\alpha}_{ij} & = & \frac{1}{2} \sum_{i,j=1}^m \, g^{ij} \sum_{s=1}^{m} g^{\alpha s} \left( g_{js,i} + g_{si,j} - g_{ij,s} \right) \\
& = &
\frac{1}{2} \sum_{i,j,s=1}^{m} g^{\alpha s} \left( g^{ij} \partial_{x_{i}} g_{js} +  g^{ij} \partial_{x_{j}} g_{si} -  g^{ij} \partial_{x_{s}} g_{ij} \right)  \\
& = & \frac{1}{2} \sum_{i,j,s=1}^{m} \left( - (\partial_{x_{i}} g^{ij})  g^{\alpha s} g_{js} -  (\partial_{x_{j}} g^{ij} ) g^{\alpha s} g_{si} -
 g^{\alpha s} g^{ij} \partial_{x_{s}} g_{ij} \right)  \\
& = & \frac{1}{2} \left( - \sum_{i=1}^m\, \partial_{x_{i}} g^{i\alpha} -  \sum_{j=1}^m \, \partial_{x_{j}} g^{\alpha j} \right) -
\frac{1}{2} \sum_{i,j,s=1}^{m} \,g^{\alpha s} g^{ij} \partial_{x_{s}} g_{ij} \\
& = & - \sum_{\gamma =1}^{m-1} \, \partial_{x_{\gamma}} g^{\gamma\alpha} - \frac{1}{2} \sum_{\beta=1}^{m-1} g^{\alpha \beta} \frac{\partial}{\partial {x_{\beta}}} \ln \big| g \big|.
\end{eqnarray*}
The second equality is obtained from (\ref{E:2.9}).
\end{proof}
\begin{corollary}  \label{Corollary:2.2}
In the boundary normal coordinate system on $U_{\epsilon_{0}}$, $\Delta_{M} + \lambda$ is expressed as follows:
\begin{eqnarray*}
\Delta_{M} + \lambda
& = & - \partial_{x_{m}}^{2} \Id ~ + ~ \left( A(x, x_{m}) - ~ 2 \omega_{m} \right)  \partial_{x_{m}} ~ + ~  D \left( x, x_{m},\frac{\partial}{\partial x}, \lambda \right)  \\
& & - \left( \partial_{x_{m}} \omega_{m} + \omega_{m} \omega_{m} - A(x, x_{m}) \omega_{m} \right) ,
\end{eqnarray*}
where $\Id$ is an $r_{0} \times r_{0} $ identity matrix and
\begin{eqnarray}   \label{E:2.10}
A(x, x_{m}) & = & \left\{ - \frac{1}{2} \sum_{\alpha, \beta = 1}^{m-1} g^{\alpha\beta}(x, x_{m}) ~ g_{\alpha\beta,m}(x, x_{m}) \right\} \Id ,  \\
D\left(x, x_{m}, \frac{\partial}{\partial x}, \lambda \right) & = &
\left\{ \left( - \sum_{\alpha, \beta = 1}^{m-1} g^{\alpha\beta}(x, x_{m}) \partial_{x_{\alpha}} \partial_{x_{\beta}} + \lambda \right)\right. \label{E:2.110}   \\
& &\left.-\sum_{\alpha, \beta = 1}^{m-1} \left( \frac{1}{2} g^{\alpha\beta}(x, x_{m}) (\partial_{x_{\alpha}} \ln |g|(x, x_{m}) ) +
(\partial_{x_{\alpha}} g^{\alpha\beta}(x, x_{m})) \right) \partial_{x_{\beta}} \right\} \Id  \nonumber  \\
& & - 2 \sum_{\alpha, \beta = 1}^{m-1} g^{\alpha \beta}(x, x_{m}) \omega_{\beta} \partial_{x_{\alpha}}
- \sum_{\alpha,\beta = 1}^{m-1} g^{\alpha\beta} (x, x_{m}) \left( \partial_{x_{\alpha}} \omega_{\beta} + \omega_{\alpha} \omega_{\beta} - \sum_{\gamma =1}^{m-1}\, \Gamma_{\alpha\beta}^{\gamma} \omega_{\gamma} \right) - E .     \nonumber
\end{eqnarray}
\end{corollary}
We extend the unit normal vector field $\partial_{x_{m}}$ to $U_{\epsilon_{0}}$ along the geodesic $\gamma_{x}(u)$, which we denote by $\partial_{x_{m}}$ again.
We recall that the geodesic $\gamma_{x}(u)$ and the corresponding coordinate $x_{m}$ flow from ${\mathcal N}^{+}$ to ${\mathcal N}^{-}$ and
$M_{0}$ is a compact manifold with boundary $\partial M \cup {\mathcal N}^{+} \cup {\mathcal N}^{-}$.
This shows that for $x^{\pm} \in {\mathcal N}^{\pm}$, $(x^{+}, r)$ with $- \epsilon_{0} < r \leq 0$ belongs to $M_{0}$ and similarly
$(x^{-}, r)$ with $0 \leq r < \epsilon_{0}$ belongs to $M_{0}$.
We keep this fact in mind and define $M_{0^{\mp}}$,  $M_{x_{m}^{-}}$ for $x_{m}^{-} < 0$ and $M_{x_{m}^{+}}$ for $x_{m}^{+} > 0$ as
\begin{eqnarray}   \label{E:2.71}
M_{0^{-}} = M_{0^{+}}  :=  M_{0}, \qquad
M_{x_{m}^{-}}  :=  M - \cup_{x_{m}^{-} < u < 0} {\mathcal N}_{u}, \qquad
M_{x_{m}^{+}}  :=  M - \cup_{0 < u < x_{m}^{+}} {\mathcal N}_{u} .
\end{eqnarray}
\noindent
We write ${\mathcal N}^{\pm}$ as ${\mathcal N}_{0}^{\pm}$.
Then, $M_{x_{m}^{-}}$ and $M_{x_{m}^{+}}$ are compact manifolds with boundary $\partial M \cup {\mathcal N}_{x_{m}^{-}} \cup {\mathcal N}_{0}^{-}$
and $\partial M \cup {\mathcal N}_{0}^{+} \cup {\mathcal N}_{x_{m}^{+}}$, respectively.
We put ${\mathcal E}_{x_{m}^{\mp}} := {\mathcal E}|_{M_{x_{m}^{\mp}}}$.
We define $Q^{\pm}_{x_{m}^{\mp}}(\lambda) : C^{\infty}({\mathcal E}|_{{\mathcal N}_{x_{m}^{\mp}}}) \rightarrow C^{\infty}({\mathcal E}|_{{\mathcal N}_{x_{m}^{\mp}}})$,
$\Psi^{\pm}_{x_{m}^{\mp}}(\lambda) : C^{\infty}({\mathcal E}|_{{\mathcal N}_{x_{m}^{\mp}}}) \rightarrow C^{\infty}({\mathcal E}|_{{\mathcal N}_{0}^{\mp}})$ as follows.
For $f^{\mp} \in C^{\infty}({\mathcal E}|_{{\mathcal N}_{x_{m}^{\mp}}})$,  choose $\psi^{\mp} \in C^{\infty}({\mathcal E}_{x_{m}^{\mp}})$  such that
\begin{eqnarray}    \label{E:2.72}
& & \left( \Delta_{M} + \lambda \right) \psi^{\mp} = 0, \qquad \psi^{\mp}|_{\partial M} = 0, \qquad
\psi^{\mp}|_{{\mathcal N}_{x_{m}^{\mp}}} = f^{\mp},  \qquad    \psi^{\mp}|_{{\mathcal N}_{0}^{\mp}} = 0.
\end{eqnarray}
\noindent
We define
\begin{eqnarray}  \label{E:2.73}
& & Q^{+}_{x_{m}^{-}}(\lambda)(f^{-}) = \left( \nabla_{{\partial}_{x_{m}}} \psi^{-} \right)|_{{\mathcal N}_{x_{m}^{-}}}, \qquad
\Psi^{+}_{x_{m}^{-}}(\lambda)(f^{-}) = \left(- \nabla_{{\partial}_{x_{m}}} \psi^{-} \right)|_{{\mathcal N}_{0}^{-}},  \\
& & Q^{-}_{x_{m}^{+}}(\lambda)(f^{+}) = \left(- \nabla_{{\partial}_{x_{m}}} \psi^{+} \right)|_{{\mathcal N}_{x_{m}^{+}}},  \qquad
\Psi^{-}_{x_{m}^{+}}(\lambda)(f^{+}) = \left(\nabla_{{\partial}_{x_{m}}} \psi^{-} \right)|_{{\mathcal N}_{0}^{+}}.   \nonumber
\end{eqnarray}
\noindent
The Dirichlet-to-Neumann operator $R(\lambda)$ defined in (\ref{E:1.2}) is
\begin{eqnarray}   \label{E:2.100}
R(\lambda) = Q^{+}_{0}(\lambda) + Q^{-}_{0}(\lambda) + \Psi^{+}_{0}(\lambda) + \Psi^{-}_{0}(\lambda).
\end{eqnarray}
\noindent
{\it Remark} : If $M_{0}$ has two components, then $\Psi^{+}_{x_{m}^{-}}(\lambda) = \Psi^{-}_{x_{m}^{+}}(\lambda) = 0$.
\begin{lemma}   \label{Lemma:2.70}
$\Psi^{\pm}_{0}(\lambda)$ are smoothing operators.
\end{lemma}
\begin{proof}
We are going to show that $\Psi^{+}_{0}(\lambda)$ is a smoothing operator. The same method works for $\Psi^{-}_{0}(\lambda)$.
For $f \in C^{\infty}({\mathcal E}|_{{\mathcal N}^{+}_{0}})$, choose $\psi \in C^{\infty}({\mathcal E}_{0})$ such that
\begin{eqnarray}   \label{E:2.76}
\left( \Delta_{M} + \lambda \right) \psi = 0, \qquad \psi|_{\partial M} = 0, \qquad \psi|_{{\mathcal N}_{0}^{+}} = f, \qquad
\psi|_{{\mathcal N}_{0}^{-}} = 0.
\end{eqnarray}
For $x_{m} \leq 0$,
\begin{eqnarray}  \label{E:2.74}
\Psi^{+}_{x_{m}}(\lambda) \left( \Psi(x, x_{m})|_{{\mathcal N}_{x_{m}}} \right) & = &
\left(- \nabla_{{\partial}_{x_{m}}} \psi \right)|_{{\mathcal N}_{0}^{-}}.
\end{eqnarray}
We note that the right hand side of (\ref{E:2.74}) does not depend on $x_{m}$ since the normal derivative is taken on
${\mathcal N}^{-}$. Taking the derivative of (\ref{E:2.74}) with respect to $x_{m}$ and using $\nabla_{\partial_{x_{m}}} = \partial_{x_{m}} + \omega_{m}$, we get
\begin{eqnarray*}
0 & = & \left( \partial_{x_{m}} \Psi^{+}_{x_{m}}(\lambda) \right) \left( \Psi(x, x_{m})|_{{\mathcal N}_{x_{m}}} \right) +
\Psi^{+}_{x_{m}}(\lambda) \left( \partial_{x_{m}} \Psi(x, x_{m}) \right)|_{{\mathcal N}_{x_{m}}}  \nonumber \\
& = & \left( \partial_{x_{m}} \Psi^{+}_{x_{m}}(\lambda) \right) \left( \Psi(x, x_{m})|_{{\mathcal N}_{x_{m}}} \right) +
\Psi^{+}_{x_{m}}(\lambda) \left( \nabla_{\partial_{x_{m}}} - \omega_{m} \right) \Psi(x, x_{m})|_{{\mathcal N}_{x_{m}}}   \nonumber   \\
& = & \left\{ \partial_{x_{m}} \Psi^{+}_{x_{m}}(\lambda)  + \Psi^{+}_{x_{m}}(\lambda) Q^{+}_{x_{m}}(\lambda)  -
\Psi^{+}_{x_{m}}(\lambda) \omega_{m} \right\} \Psi(x, x_{m})|_{{\mathcal N}_{x_{m}}}.
\end{eqnarray*}
\noindent
When $x_{m} = 0$, the above equality leads to
\begin{eqnarray}    \label{E:2.75}
\Psi^{+}_{0}(\lambda) Q^{+}_{0}(\lambda) & = & - \partial_{x_{m}} \Psi^{+}_{0}(\lambda) + \Psi^{+}_{0}(\lambda) \omega_{m}.
\end{eqnarray}
\noindent
We note that
$\Psi^{+}_{0}(\lambda)$, $\partial_{x_{m}} \Psi^{+}_{0}(\lambda)$, $\Psi^{+}_{0}(\lambda) ~ \omega_{m}$ and $Q^{+}_{0}(\lambda)$ are $\Psi$DO's of order $1$, and $Q^{+}_{0}(\lambda)$ is an elliptic operator.
If we consider the homogeneous symbol of order $2$ in (\ref{E:2.75}), we can see that the order $1$ part of the homogeneous symbol of $\Psi^{+}_{0}(\lambda)$ is zero. Using (\ref{E:2.75}) repeatedly, one can show recursively that
all parts of the homogeneous symbol of $\Psi^{+}_{0}(\lambda)$ vanish.
Hence, $\Psi^{+}_{0}(\lambda)$ is a smoothing operator.
\end{proof}
The above result with (\ref{E:2.100}) shows that it is enough to consider only $Q_{0}^{\pm}(\lambda)$
when we compute the homogeneous symbol of $R(\lambda)$.
To compute the homogeneous symbols of $Q^{\pm}_{x_{m}}(\lambda)$, we are going to find two Riccati type equations involving
$Q^{\pm}_{x_{m}}(\lambda)$ as follows.
This idea goes back to I.M. Gelfand.
We recall that for $f \in C^{\infty}({\mathcal E}|_{{\mathcal N}_{0}})$,
we choose $\psi(x, x_{m})$ satisfying the property (\ref{E:2.76}).
Using $\nabla_{\partial_{x_{m}}} = \partial_{x_{m}} + \omega_{m}$ and the definition of $Q^{+}_{x_{m}}(\lambda)$ for $x_{m} \leq 0$, we get
\begin{eqnarray*}
\left( (\partial_{x_{m}} + \omega_{m}) \psi \right)\big|_{{\mathcal N}_{x_{m}}}  & = &
\left( \nabla_{\partial_{x_{m}}} \psi \right)\big|_{{\mathcal N}_{x_{m}}}
 ~ = ~ Q^{+}_{x_{m}}(\lambda) \left( \psi|_{{\mathcal N}_{x_{m}}} \right) .
\end{eqnarray*}
\noindent
Taking $\nabla_{\partial_{m}} = \partial_{m} + \omega_{m}$ one more time, we have
\begin{eqnarray*}
\left\{ ( \partial_{x_{m}} + \omega_{m} )^{2} \psi \right\}\big|_{{\mathcal N}_{x_{m}}} & = &
(\partial_{x_{m}} + \omega_{m}) Q^{+}_{x_{m}}(\lambda) \left( \psi|_{{\mathcal N}_{x_{m}}} \right)   \\
& = & ( \partial_{x_{m}} Q^{+}_{x_{m}}(\lambda) )  \left( \psi|_{{\mathcal N}_{x_{m}}} \right) +
Q^{+}_{x_{m}}(\lambda) \left( \partial_{x_m} \psi|_{{\mathcal N}_{x_{m}}} \right)
+ \omega_{m} Q^{+}_{x_{m}}(\lambda) \left( \psi|_{{\mathcal N}_{x_{m}}} \right)   \\
& = & \left( \partial_{x_{m}} Q^{+}_{x_{m}}(\lambda) + Q^{+}_{x_{m}}(\lambda)^{2}
 - Q^{+}_{x_{m}}(\lambda) \omega_{m} + \omega_{m} Q^{+}_{x_{m}}(\lambda) \right) \left( \psi|_{{\mathcal N}_{x_{m}}} \right).
\end{eqnarray*}
\noindent
On the other hand, Corollary \ref{Corollary:2.2} with (\ref{E:2.76}) shows that
\begin{eqnarray*}
\left\{ ( \partial_{x_{m}} + \omega_{m} )^{2} \psi \right\}\big|_{{\mathcal N}_{x_{m}}} & = &
\left\{ ( \partial_{x_{m}}^{2} + (\partial_{x_{m}} \omega_{m}) + 2 \omega_{m} \partial_{x_{m}} +
\omega_{m} \omega_{m})  \psi \right\}\big|_{{\mathcal N}_{x_{m}}}  \\
& = & \left\{\left(D\left(x, x_{m}, \frac{\partial}{\partial x}, \lambda\right) + A(x, x_{m}) \partial_{x_{m}} +
A(x, x_{m}) \omega_{m}\right) \psi \right\}\big|_{{\mathcal N}_{x_{m}}}   \\
& = & \left\{ \left(D\left( x, x_{m}, \frac{\partial}{\partial x}, \lambda \right) + A(x, x_{m}) Q^{+}_{x_{m}}(\lambda)\right) \psi \right\}\big|_{{\mathcal N}_{x_{m}}},
\end{eqnarray*}
\noindent
which leads to
\begin{eqnarray}   \label{E:2.14}
Q^{+}_{x_{m}}(\lambda)^{2}  +  \partial_{x_{m}} Q^{+}_{x_{m}}(\lambda) - Q^{+}_{x_{m}}(\lambda) \omega_{m} + \omega_{m} Q^{+}_{x_{m}}(\lambda)  =
D\left( x, x_{m}, \frac{\partial}{\partial x}, \lambda \right) + A(x, x_{m}) Q^{+}_{x_{m}}(\lambda).
\end{eqnarray}
We next consider the case of $Q^{-}_{x_{m}}(\lambda)$ for $x_{m} \geq 0$ in the same way.
For $f \in C^{\infty}({\mathcal E}_{{\mathcal N}_{0}^{-}})$, we choose $\phi(x, x_{m}) \in C^{\infty}({\mathcal E}_{0})$ satisfying
\begin{eqnarray*}
\left( \Delta_{M} + \lambda \right) \phi = 0, \qquad \phi|_{\partial M} = 0, \qquad \phi|_{{\mathcal N}_{0}^{+}} = 0, \qquad
\phi|_{{\mathcal N}_{0}^{-}} = f.
\end{eqnarray*}
\noindent
From the definition of $Q^{-}_{x_{m}}(\lambda)$, we get
\begin{eqnarray*}
\left( \nabla_{\partial_{x_{m}}} \phi \right)\big|_{{\mathcal N}_{x_{m}}} & = &
\left( (\partial_{x_{m}} + \omega_{m}) \phi \right)\big|_{{\mathcal N}_{x_{m}}}
 ~ = ~ - ~ Q^{-}_{x_{m}}(\lambda) \left( \phi|_{{\mathcal N}_{x_{m}}} \right) ,
\end{eqnarray*}
\noindent
which leads to
\begin{eqnarray*}
\left\{ ( \partial_{x_{m}} + \omega_{m} )^{2} \phi \right\}\big|_{{\mathcal N}_{x_{m}}} & = &
- (\partial_{x_{m}} + \omega_{m}) Q^{-}_{x_{m}}(\lambda) \left( \phi|_{{\mathcal N}_{x_{m}}} \right)   \\
& = & - ( \partial_{x_{m}} Q^{-}_{x_{m}}(\lambda) )  \left( \phi|_{{\mathcal N}_{x_{m}}} \right) -
Q^{-}_{x_{m}}(\lambda) \left( \partial_{x_{m}} \phi|_{{\mathcal N}_{x_{m}}} \right)
- \omega_{m} Q^{-}_{x_{m}}(\lambda) \left( \phi|_{{\mathcal N}_{x_{m}}} \right)   \\
& = & \left( - \partial_{x_{m}} Q^{-}_{x_{m}}(\lambda) + Q^{-}_{x_{m}}(\lambda)^{2}
 + Q^{-}_{x_{m}}(\lambda) \omega_{m} - \omega_{m} Q^{-}_{x_{m}}(\lambda) \right) \left( \phi|_{{\mathcal N}_{x_{m}}} \right).
\end{eqnarray*}
\noindent
On the other hand, Corollary \ref{Corollary:2.2} shows that
\begin{eqnarray*}
\left\{ ( \partial_{x_{m}} + \omega_{m} )^{2} \phi \right\}\big|_{{\mathcal N}_{x_{m}}} & = &
\left\{ (\partial_{x_{m}}^{2} + (\partial_{x_{m}} \omega_{m}) +
2 \omega_{m} \partial_{x_{m}} + \omega_{m} \omega_{m})  \phi \right\}\big|_{{\mathcal N}_{x_{m}}}  \\
& = & \left\{ \left(D\left( x, x_{m}, \frac{\partial}{\partial x}, \lambda\right) +
A(x, x_{m}) \partial_{x_{m}} + A(x, x_{m}) \omega_{m}\right) \phi \right\}\big|_{{\mathcal N}_{x_{m}}}   \\
& = & \left\{ \left(D\left( x, x_{m}, \frac{\partial}{\partial x}, \lambda \right) - A(x, x_{m}) Q^{-}_{x_{m}}(\lambda)\right) \phi \right\}\big|_{{\mathcal N}_{x_{m}}},
\end{eqnarray*}
\noindent
which leads to
\begin{eqnarray}   \label{E:2.15}
Q^{-}_{x_{m}}(\lambda)^{2}  -  \partial_{x_{m}} Q^{-}_{x_{m}}(\lambda) - \omega_{m} Q^{-}_{x_{m}}(\lambda) + Q^{-}_{x_{m}}(\lambda) \omega_{m}  =
D\left( x, x_{m}, \frac{\partial}{\partial x}, \lambda\right) - A(x, x_{m}) Q^{-}_{u}(\lambda).
\end{eqnarray}
\noindent
Eqs.~(\ref{E:2.14}) and (\ref{E:2.15}) give the following two Riccati type equations,
which are very useful in computing the homogeneous symbols of $Q^{\pm}_{x_{m}}(\lambda)$.
\begin{lemma}   \label{Lemma:2.3}
$Q^{\pm}_{x_{m}}(\lambda)$ satisfy the following equations.
\begin{eqnarray*}
Q^{+}_{x_{m}}(\lambda)^{2} & = & D\left( x, x_{m},\frac{\partial}{\partial x}, \lambda\right)
+ A(x, x_{m}) Q^{+}_{x_{m}}(\lambda) - \partial_{x_{m}} Q^{+}_{x_{m}}(\lambda) + \left( Q^{+}_{x_{m}}(\lambda) \omega_{m} - \omega_{m} Q^{+}_{x_{m}}(\lambda)\right) ,        \nonumber    \\
Q^{-}_{x_{m}}(\lambda)^{2} & = & D\left( x, x_{m}, \frac{\partial}{\partial x}, \lambda\right)
- A(x, x_{m}) Q^{-}_{x_{m}}(\lambda) + \partial_{x_{m}} Q^{-}_{x_{m}}(\lambda) - \left( Q^{-}_{x_{m}}(\lambda) \omega_{m} - \omega_{m} Q^{-}_{x_{m}}(\lambda)\right) .
\end{eqnarray*}
\end{lemma}
\noindent
{\it Remark} :
(1) We may identify ${\mathcal N}_{x_{m}}$ with ${\mathcal N}_{0}$ along the geodesic $\gamma_{x}(u)$.
Using this identification, we may regard $Q^{\pm}_{x_{m}}(\lambda)$ to be one parameter families of operators defined on $C^{\infty}({\mathcal E}|_{{\mathcal N}_{0}})$.\\
\noindent
(2) When $M$ is a closed manifold and $M_{0}$ has two components $M_{1}$ and $M_{2}$,
$Q_{0}^{+}(0)$ is equal to the Dirichlet-to-Neumann operator defined in \cite{LU} or \cite{PS}.
In this case, the first equation in the above lemma is equal to the equation given in (1.4) of \cite{LU}.

We now use Lemma \ref{Lemma:2.3} to compute the homogeneous symbols of $Q^{\pm}_{x_{m}}(\lambda)$.
We denote the homogeneous symbols of $D \left( x, x_{m},\frac{\partial}{\partial x} ,\lambda \right)$,
$Q^{+}_{x_{m}}(\lambda)$ and $Q^{-}_{x_{m}}(\lambda)$ as follows (cf. (\ref{E:2.110})).
\begin{eqnarray}   \label{E:2.16}
\sigma \left(D\left(  x, x_{m},\frac{\partial}{\partial x},  \lambda\right) \right) & = & p_{2}(x, x_{m}, \xi, \lambda) ~ + ~
 p_{1}(x, x_{m}, \xi) ~ + ~ p_{0}(x, x_{m}, \xi),    \\
\sigma \left( Q^{+}_{x_{m}}(\lambda) \right) & \sim & \alpha_{1}(x, x_{m}, \xi, \lambda) ~ + ~
\alpha_{0}(x, x_{m}, \xi, \lambda) ~ + ~ \alpha_{-1}(x, x_{m}, \xi, \lambda) ~ + ~  \cdots,   \nonumber  \\
\sigma \left( Q^{-}_{x_{m}}(\lambda) \right) & \sim & \beta_{1}(x, x_{m}, \xi, \lambda) ~ + ~
\beta_{0}(x, x_{m}, \xi, \lambda) ~ + ~ \beta_{-1}(x, x_{m}, \xi, \lambda) ~ + ~  \cdots,   \nonumber
\end{eqnarray}
\noindent
where
\begin{eqnarray}   \label{E:2.17}
& & p_{2}(x, x_{m}, \xi, \lambda) ~ = ~ \left( \sum_{\alpha, \beta = 1}^{m-1} g^{\alpha\beta}(x,  x_{m}) \xi_{\alpha} \xi_{\beta} + \lambda \right) \Id
 ~ = ~ \left( | \xi |^{2} + \lambda \right) \Id,   \\
& & p_{1}(x,  x_{m}, \xi) ~ = ~ - i
\sum_{\alpha, \beta = 1}^{m-1} \left( \frac{1}{2} g^{\alpha\beta}(x,  x_{m}) \partial_{x_{\alpha}} \ln |g|(x,  x_{m}) +
\partial_{x_{\alpha}} g^{\alpha\beta}(x,  x_{m}) \right) \xi_{\beta} \Id
- 2 i \sum_{\alpha, \beta = 1}^{m-1} g^{\alpha\beta} \omega_{\alpha} \xi_{\beta} ,   \nonumber \\
& & p_{0}(x,  x_{m}, \xi) ~ = ~
- \sum_{\alpha \beta =1}^{m-1} \,  g^{\alpha ,\beta} \left( \partial_{x_{\alpha}} \omega_{\beta} + \omega_{\alpha} \omega_{\beta} - \sum_{\gamma =1}^{m-1} \,\Gamma_{\alpha\beta}^{\gamma} \omega_{\gamma} \right) - E . \nonumber
\end{eqnarray}
\noindent
Then, we have
\begin{eqnarray}   \label{E:2.18}
\sigma \left( \partial_{x_{m}} Q^{+}_{x_{m}}(\lambda) \right) & = & \partial_{x_{m}} \alpha_{1}(x, x_{m}, \xi, \lambda) +
\partial_{x_{m}} \alpha_{0}(x, x_{m}, \xi, \lambda)  +  \partial_{x_{m}} \alpha_{-1}(x,  x_{m}, \xi, \lambda)  +   \cdots ,     \\
\sigma \left( \partial_{x_{m}} Q^{-}_{x_{m}}(\lambda) \right) & = & \partial_{x_{m}} \beta_{1}(x,  x_{m}, \xi, \lambda)  +
\partial_{x_{m}} \beta_{0}(x, x_{m}, \xi, \lambda)  +  \partial_{x_{m}} \beta_{-1}(x,  x_{m}, \xi, \lambda)  +   \cdots .   \nonumber
\end{eqnarray}
\noindent
We put $D_{x} := \frac{1}{i} \partial_{x}$ and denote by $(x, \xi)$ a coordinate for $T^{\ast}{\mathcal N}$.
It is well known \cite{Gi1,Sh} that
\begin{eqnarray}   \label{E:2.19}
\sigma \left( Q^{+}_{x_{m}}(\lambda)^{2} \right) & \sim & \sum_{k=0}^{\infty} \sum_{\stackrel{|\omega|+i+j=k}{i, j \geq 0}} \frac{1}{\omega !} \partial^{\omega}_{\xi} \alpha_{1-i}(x, x_{m}, \xi, \lambda) \cdot D_{x}^{\omega}\alpha_{1-j}(x, x_{m}, \xi, \lambda)       \\
& = & \alpha_{1}^{2} + \left(\partial_{\xi} \alpha_{1} \cdot D_{x} \alpha_{1} + 2 \alpha_{1} \cdot \alpha_{0} \right)     \nonumber \\
& & + ~ \left( 2 \alpha_{1} \alpha_{-1} + \alpha_{0}^{2} - i (\partial_{\xi} \alpha_{0} ) (\partial_{x} \alpha_{1} )-
i (\partial_{\xi} \alpha_{1} )(\partial_{x} \alpha_{0}) - \sum_{| \omega | =2}\,
\frac{1}{\omega !} (\partial^{\omega}_{\xi} \alpha_{1})( \partial^{\omega}_{x} \alpha_{1}) \right) ~ + ~ \cdots , \nonumber \\
\sigma \left( Q^{-}_{x_{m}}(\lambda)^{2} \right) & \sim & \sum_{k=0}^{\infty} \sum_{\stackrel{|\omega|+i+j=k}{i, j \geq 0}} \frac{1}{\omega !} \partial^{\omega}_{\xi} \beta_{1-i}( x,x_m, \xi, \lambda) \cdot
D_{x}^{\omega}\beta_{1-j}(x,x_m, \xi, \lambda)   \\
& = & \beta_{1}^{2} + \left(\partial_{\xi} \beta_{1} \cdot D_{x} \beta_{1} + 2 \beta_{1} \cdot \beta_{0} \right) \nonumber \\
& & + ~ \left( 2 \beta_{1} \beta_{-1} + \beta_{0}^{2} - i (\partial_{\xi} \beta_{0} )(\partial_{x} \beta_{1}) -
i (\partial_{\xi} \beta_{1} )(\partial_{x} \beta_{0}) - \sum_{|\omega |=2}
\frac{1}{\omega !} (\partial^{\omega}_{\xi} \beta_{1})( \partial^{\omega}_{x} \beta_{1}) \right) ~ + ~ \cdots .  \nonumber
\end{eqnarray}
\noindent
By comparing the degree of homogeneity of the equations in Lemma \ref{Lemma:2.3},
we can compute the homogeneous symbols of $Q^{+}_{x_{m}}(\lambda)$ and $Q^{-}_{x_{m}}(\lambda)$.
For example, the first three terms are given as follows (cf. (2.2) and (2.3) in \cite{PS}).
From now on we often omit $\Id$ to simplify the presentation.
\begin{eqnarray}   \label{E:2.20}
& & \alpha_{1}(x, x_{m}, \xi, \lambda) ~ = ~ \sqrt{|\xi|^2 + \lambda}, \\
& & \alpha_{0}(x, x_{m}, \xi, \lambda) ~ = ~ \frac{1}{2 \sqrt{|\xi|^2 + \lambda}}
\left\{ p_{1}(x, x_{m}, \xi) + A(x, x_{m}) \sqrt{|\xi|^2 + \lambda}  \right.     \nonumber  \\
& & \hspace{1.0 cm} \left. - \partial_{\xi} \sqrt{|\xi|^2 + \lambda} \cdot \frac{1}{i} \partial_{x}
\sqrt{|\xi|^2 + \lambda} - \partial_{x_{m}} \sqrt{|\xi|^2 + \lambda} \right\},    \nonumber  \\
& & \alpha_{-1}(x, x_{m}, \xi, \lambda) ~ = ~ \frac{1}{2 \sqrt{|\xi|^2 + \lambda}}
\Big\{ \sum_{| \omega |=2} \frac{1}{\omega !} (\partial_{\xi}^{\omega} \alpha_{1} )(\partial_{x}^{\omega} \alpha_{1} )+
i (\partial_{\xi} \alpha_{0})( \partial_{x} \alpha_{1} )+ i (\partial_{\xi} \alpha_{1})( \partial_{x} \alpha_{0})
- \alpha_{0}^{2}   \nonumber  \\
& & \hspace{1.0 cm} ~ +p_{0} + A(x, x_{m}) \alpha_{0} - \partial_{x_{m}} \alpha_{0} + (\alpha_{0} \omega_{m} - \omega_{m} \alpha_{0}) \Big\}.   \nonumber
\end{eqnarray}
\noindent
{\it Remark} : When $\lambda = 0$ and $x_{m} = 0$, the homogeneous symbol of $Q_{0}^{+}(0)$ given above is equal to the one for the Dirichlet-to-Neumann operator given in
(2.2) - (2.4) of \cite{PS}.\\
\noindent
Similarly,
\begin{eqnarray}  \label{E:2.21}
\beta_{1}(x, x_{m}, \xi, \lambda) & = & \sqrt{|\xi|^2 + \lambda},    \\
\beta_{0}(x, x_{m}, \xi, \lambda) & = & \frac{1}{2 \sqrt{|\xi|^2 + \lambda}}
\left\{ p_{1}(x, x_{m}, \xi) - A(x, x_{m}) \sqrt{|\xi|^2 + \lambda}  \right.    \nonumber \\
& & \left. - \partial_{\xi} \sqrt{|\xi|^2 + \lambda} \cdot \frac{1}{i} \partial_{x}
\sqrt{| \xi|^2 + \lambda} + \partial_{x_{m}} \sqrt{| \xi|^2 + \lambda} \right\},    \nonumber  \\
\beta_{-1}(x, x_m, \xi, \lambda) & = & \frac{1}{2 \sqrt{|\xi|^2 + \lambda}}
\Big\{ \sum_{| \omega | =2} \frac{1}{\omega !} (\partial_{\xi}^{\omega} \beta_{1} )(\partial_{x}^{\omega} \beta_{1} )+
i (\partial_{\xi} \beta_{0})( \partial_{x} \beta_{1}) + i (\partial_{\xi} \beta_{1})( \partial_{x} \beta_{0})
- \beta_{0}^{2}   \nonumber  \\
& & ~ + p_{0} - A(x, x_{m}) \beta_{0} + \partial_{x_{m}} \beta_{0} - ( \beta_{0} \omega_{m} - \omega_{m} \beta_{0}) \Big\} .  \nonumber
\end{eqnarray}
\noindent
Since $R(\lambda) = Q_{0}^{+}(\lambda) + Q^{-}_{0}(\lambda)$ up to a smoothing operator, the homogeneous symbol of $R(\lambda)$ is given as follows.
\begin{eqnarray}    \label{E:2.22}
\sigma(R(\lambda))(x, 0, \xi, \lambda)  & \sim & \theta_{1}(x, 0, \xi, \lambda)
+ \theta_{0}(x, 0, \xi, \lambda) + \theta_{-1}(x, 0, \xi, \lambda) + \cdots,
\end{eqnarray}
\noindent
where $\theta_{1-j}(x, 0, \xi, \lambda) = \alpha_{1-j}(x, 0, \xi, \lambda) + \beta_{1-j}(x, 0, \xi, \lambda)$.
The first three terms of $\sigma(R(\lambda))(x, 0, \xi, \lambda)$ are given as follows.
\begin{eqnarray}    \label{E:2.23}
\theta_{1} & = & \alpha_{1} + \beta_{1} ~ = ~ 2 \sqrt{|\xi|^2 + \lambda},   \\
\theta_{0} & = & \alpha_{0} + \beta_{0} ~ = ~ \frac{1}{\sqrt{|\xi |^2 + \lambda}}
\left\{ p_{1}  - \partial_{\xi} \sqrt{|\xi|^2 + \lambda} \cdot \frac{1}{i} \partial_{x} \sqrt{|\xi|^2 + \lambda} \right\},  \nonumber \\
\theta_{-1} & = & \alpha_{-1} + \beta_{-1} ~ = ~ \frac{1}{2 \sqrt{|\xi|^2 + \lambda}} \Big\{\sum_{| \omega | =2} \frac{2}{\omega !} (\partial_{\xi}^{\omega} \alpha_{1} )(\partial_{x}^{\omega} \alpha_{1} )~ + ~
i \partial_{\xi} ( \alpha_{0} + \beta_{0} ) \partial_{x} \alpha_{1} +
i (\partial_{\xi} \alpha_{1}) \partial_{x} ( \alpha_{0} + \beta_{0} )  \nonumber   \\
& &  ~ - ( \alpha_{0}^{2} + \beta_{0}^{2}) + 2 p_{0} + A(x,0) (\alpha_{0} - \beta_{0}) - \partial_{x_{m}} (\alpha_{0} - \beta_{0}) \Big\},  \nonumber
\end{eqnarray}
\noindent
where
\begin{eqnarray}   \label{E:2.24}
(\alpha_{0} - \beta_{0})(x, x_{m}, \xi, \lambda) & = & A(x, x_{m}) ~ - ~
\frac{\partial_{x_{m}} \sqrt{|\xi|^2 + \lambda}}{\sqrt{|\xi|^2 + \lambda}},
\end{eqnarray}
\noindent
and hence $(\alpha_{0} - \beta_{0}) \omega_{m} - \omega_{m} (\alpha_{0} - \beta_{0}) = 0$ since $\alpha_{0} - \beta_{0}$ is a scalar function.

We denote the homogeneous symbol of the resolvent $\left( \mu - R(\lambda) \right)^{-1}$ by
\begin{eqnarray}    \label{E:2.25}
\sigma \left( (\mu - R(\lambda))^{-1} \right)(x, \xi, \lambda, \mu) & \sim & r_{-1}(x, \xi, \lambda, \mu) + r_{-2}(x, \xi, \lambda, \mu) +
r_{-3}(x, \xi, \lambda, \mu) + \cdots.
\end{eqnarray}
\noindent
Then,
\begin{eqnarray}    \label{E:2.26}
r_{-1}(x, \xi, \lambda, \mu) & = & \left( \mu - 2 \sqrt{|\xi|^2 + \lambda} \right)^{-1},  \\
r_{-1-j}(x, \xi, \lambda, \mu) & = & \left( \mu - 2 \sqrt{|\xi|^2 + \lambda} \right)^{-1}\,\, \sum_{k=0}^{j-1} \,\,\sum_{|\omega|+l+k=j} \frac{1}{\omega!}
\partial_{\xi}^{\omega} \theta_{1-l}\cdot D_{x}^{\omega} r_{-1-k},    \nonumber
\end{eqnarray}
\noindent
which shows that the first three terms are given as follows:
\begin{eqnarray}    \label{E:2.27}
r_{-1} & = & \left( \mu - 2 \sqrt{|\xi|^2 + \lambda} \right)^{-1},  \\
r_{-2} & = & \left( \mu - 2 \sqrt{|\xi|^2 + \lambda} \right)^{-1}
\left\{ \partial_{\xi} \theta_{1} \cdot D_{x} r_{-1} + \theta_{0} \cdot r_{-1} \right\} , \nonumber  \\
r_{-3} & = & \left( \mu - 2 \sqrt{|\xi|^2 + \lambda} \right)^{-1} \Big\{  \sum_{| \omega | =2} \frac{1}{\omega !} \partial_{\xi}^{\omega} \theta_{1} \cdot D_{x}^{\omega} r_{-1}
+ \partial_{\xi} \theta_{1} \cdot D_{x} r_{-2}\nonumber\\
& &\hspace{5.0cm} + \partial_{\xi} \theta_{0} \cdot D_{x} r_{-1}  + \theta_{0} \cdot r_{-2}
+ \theta_{-1} \cdot r_{-1} \Big\}.   \nonumber
\end{eqnarray}
It was shown in the Appendix of \cite{BFK} that the coefficients $\pi_{j}$ and $q_{j}$ in (\ref{E:1.7}) are computed by the following integrals:
\begin{eqnarray}    \label{E:2.28}
\pi_{j} & = & - \frac{\partial}{\partial s}\bigg|_{s=0} \int_{{\mathcal N}}  \left( \frac{1}{(2 \pi)^{m-1}} \int_{T_{x}^{\ast}{\mathcal N}}
\frac{1}{2 \pi i} \int_{\gamma} \mu^{-s} \Tr r_{-1-j} \left( x, \xi, \frac{\lambda}{|\lambda|}, \mu\right) d\mu d \xi \right)  d \vol({\mathcal N}) ,       \\
q_{j} & = & \frac{1}{2} ~ \int_{{\mathcal N}} \left( \frac{1}{(2 \pi)^{m-1}} \int_{T_{x}^{\ast}{\mathcal N}}
\frac{1}{2 \pi i} \int_{\gamma} \mu^{-s} \Tr r_{-1-j} \left(x, \xi, \frac{\lambda}{|\lambda|}, \mu\right) d\mu d \xi \right) d \vol({\mathcal N})  \bigg|_{s=0}. \nonumber
\end{eqnarray}
\noindent
In the next section, we are going to compute $\pi_{j}$ and $q_{j}$ in the boundary normal coordinate system when ${\mathcal N}$ is a $2$-dimensional closed manifold.
\section{The asymptotic expansion of $\ln \Det R(\lambda)$ for $\lambda \rightarrow \infty$}
For $p \in {\mathcal N}$, we choose a normal coordinate $x = ( x_{1}, \cdots, x_{m-1})$ on an open neighborhood of $p$ with
$p = (0, \cdots, 0)$. We next choose the boundary normal coordinate system $(x, x_{m}) = (x_{1}, \cdots, x_{m-1}, x_{m})$ on $U_{\epsilon_{0}}$ as introduced in the beginning of Section 2.
Then, $\{ \partial_{x_{1}}, \cdots, \partial_{x_{m-1}}, \partial_{x_{m}} \}$ is a local basis for $TM\big|_{U_{\epsilon_{0}}}$.
We denote by $\nabla^{LC}$ the Levi-Civita connection on $(M, g)$ and define the Riemann curvature tensor by
\begin{eqnarray}   \label{E:3.2}
 R(\partial_{x_{i}}, \partial_{x_{j}}) \partial_{x_{k}} & := & \nabla^{LC}_{\partial_{x_{i}}} \nabla^{LC}_{\partial_{x_{j}}} \partial_{x_{k}} - \nabla^{LC}_{\partial_{x_{j}}} \nabla^{LC}_{\partial_{x_{i}}} \partial_{x_{k}} -
 \nabla^{LC}_{[\partial_{x_{i}}, \partial_{x_{j}}]} \partial_{x_{k}}.
\end{eqnarray}
\noindent
We note that at the point $p \in {\mathcal N}$,
$\{ \partial_{x_{1}}|_{p}, \cdots, \partial_{x_{m-1}}|_{p}, \partial_{x_{m}}|_{p} \}$ is an orthonormal basis for $T_{p}M$.
In this case, we define the components of the Riemannian curvature tensor and Ricci tensor at $p \in {\mathcal N}$ as follows.
\begin{eqnarray}    \label{E:3.3}
R_{\alpha\beta\gamma\delta}(p) & := & \langle R(\partial_{x_{\alpha}}, \partial_{x_{\beta}}) \partial_{x_{\gamma}}, \partial_{x_{\delta}} \rangle|_{p}, \qquad
R_{\alpha\beta}(p) ~ := ~ \sum_{\gamma=1}^{m-1} \langle R(\partial_{x_{\gamma}}, \partial_{x_{\alpha}}) \partial_{x_{\beta}}, \partial_{x_{\gamma}} \rangle|_{p}.
\end{eqnarray}
\noindent
The scalar curvatures $\tau_{{\mathcal N}}(p)$ of ${\mathcal N}$ and $\tau_{M}(p)$ of $M$ at $p$ are defined by
\begin{eqnarray}   \label{E:3.4}
\tau_{{\mathcal N}}(p) & := & \sum_{\alpha=1}^{m-1} R_{\alpha\alpha}(p)
~ = ~ \sum_{\alpha, \gamma=1}^{m-1} R_{\gamma\alpha\alpha\gamma}(p) ~ = ~ - ~ \sum_{\alpha, \gamma=1}^{m-1} R_{\gamma\alpha\gamma\alpha}(p),  \\
\tau_{M}(p) & := & \sum_{i,j=1}^{m}  \langle R(\partial_{x_{i}}, \partial_{x_{j}}) \partial_{x_{j}}, \partial_{x_{i}} \rangle|_{p} .   \nonumber
\end{eqnarray}
\noindent
It is well known \cite{PS,Vi} that for $1 \leq \alpha, \beta \leq m-1$,
\begin{eqnarray}     \label{E:3.6}
g_{\alpha\beta}(x,0) & = & \delta_{\alpha\beta} + \sum_{\mu, \nu = 1}^{m-1} \frac{1}{3} R_{\alpha\mu\beta\nu} x^{\mu} x^{\nu} + O(|x|^{3}),  \\
g^{\alpha\beta}(x,0) & = & \delta_{\alpha\beta} - \sum_{\mu, \nu = 1}^{m-1} \frac{1}{3} R_{\alpha\mu\beta\nu} x^{\mu} x^{\nu} + O(|x|^{3}),     \nonumber   \\
|g|(x,0) & = & 1 - \frac{1}{3} \sum_{\mu, \nu = 1}^{m-1} R_{\mu\nu} x^{\mu} x^{\nu} + O(|x|^{3}),  \nonumber
\end{eqnarray}
\noindent
which shows that
\begin{eqnarray}     \label{E:3.7}
g^{\alpha\beta,\gamma\epsilon}(p) & = & - ~ \frac{1}{3} ( R_{\alpha\gamma\beta\epsilon}(p) + R_{\alpha\epsilon\beta\gamma}(p)),\qquad
\partial_{x^{\mu}} \partial_{x^{\nu}} |g|(p) ~ = ~ - \frac{2}{3} R_{\mu\nu}(p).
\end{eqnarray}
Since $\partial_{x_{m}}$ is an inward normal vector field on ${\mathcal N}^{-} \subset M_{0}$, we are going to define the principal curvatures of ${\mathcal N}^{-} \subset M_{0}$. The principal curvatures of ${\mathcal N}^{+} \subset M_{0}$ are those of
${\mathcal N}^{-} \subset M_{0}$ with opposite sign.
We define the shape operator $S_{p} : T_{p} {\mathcal N}^{-} \rightarrow T_{p}{\mathcal N}^{-}$ as follows:
\begin{eqnarray}    \label{E:3.9}
S_{p}(\partial_{x_{\alpha}}|_{p}) & := & - ~ (\nabla^{LC}_{\partial_{x_{\alpha}}} \partial_{x_{m}})\big|_{p} ~ = ~
- \sum_{j = 1}^{m} \Gamma_{\alpha m}^{j}(p) ~ \partial_{x_{j}}|_{p}          \nonumber \\
& = & - \frac{1}{2} \sum_{j, s = 1}^{m} g^{js}(p) ( g_{m s, \alpha} + g_{s \alpha, m} - g_{\alpha m, s})(p) ~ \partial_{x_{j}}|_{p}   \nonumber  \\
& = & - \frac{1}{2} \sum_{j=1}^{m} g_{j \alpha, m}(p)  ~ \partial_{x_{j}}|_{p}
~ = ~ - \frac{1}{2} \sum_{\beta=1}^{m-1} g_{\alpha\beta, m}(p)  ~ \partial_{x_{\beta}}|_{p}  ,
\end{eqnarray}
\noindent
where we used the relation (\ref{E:2.33}) with $g^{js}(p) = \delta_{js}$ and  $g_{mj,k}(x, x_{m}) = 0$.
Since $(g_{\alpha\beta, m}(p))$ is a symmetric matrix, it is diagonalizable.
We initially choose a normal coordinate system on an open neighborhood of $p$ in ${\mathcal N}$ such that
\begin{eqnarray}   \label{E:3.10}
g_{\alpha \beta, m}(p) & = & \begin{cases} -2 \kappa_{\alpha} & \quad \text{for} \quad \alpha = \beta \\
0 & \quad \text{for} \quad \alpha \neq \beta . \end{cases}
\end{eqnarray}
\noindent
Then, $S_{p}(\partial_{x_{\alpha}}|_{p}) = \kappa_{\alpha} \partial_{x_{\alpha}}|_{p}$, where $\kappa_{\alpha}$ is a principal curvature of ${\mathcal N}^-$
embedded in $M$ at $p$.
Since $g^{\alpha\beta,m}(p) = - g_{\alpha\beta, m}(p)$ \cite{PS}, we get $g^{\alpha \beta, m}(p) = 2 \kappa_{\alpha} \delta_{\alpha\beta}$.
We define the $r$-mean curvature $H_{r}(p)$ of ${\mathcal N}^-$ by
\begin{eqnarray*}
H_{r}(p) & = & \frac{1}{\binom{m-1}{r}} \sigma_{r}(\kappa_{1}, \cdots, \kappa_{m-1})
~ = ~ \frac{r ! (m-1-r)!}{(m-1)!} \sigma_{r}(\kappa_{1}, \cdots, \kappa_{m-1}),
\end{eqnarray*}
\noindent
where $\sigma_{r} : {\mathbb R}^{m-1} \rightarrow {\mathbb R}$ is the $r$-th elementary symmetric function defined by
$\sigma_{r}(u_{1}, \cdots. u_{m-1}) = \sum_{1 \leq i_{1} < \cdots < i_{r} \leq m-1} u_{i_{1}} \cdots u_{i_{r}}$ \cite{ALM}.
In particular,
$H_{1}(p)$ and $H_{2}(p)$ are defined by
\begin{eqnarray}    \label{E:3.11}
H_{1}(p) & = & \frac{1}{m-1} \sum_{\alpha = 1}^{m-1} \kappa_{\alpha},
\qquad H_{2}(p) ~ = ~ \frac{2}{(m-1)(m-2)} \sum_{1 \leq \alpha < \beta \leq m-1} \kappa_{\alpha} \kappa_{\beta}.
\end{eqnarray}
\noindent
It was shown in \cite{PS} that
\begin{eqnarray}     \label{E:3.12}
& & \sum_{\alpha=1}^{m-1} g^{\alpha\alpha,m m}(p) ~ = ~ 8 \sum_{\alpha=1}^{m-1} \kappa_{\alpha}^{2} -
\sum_{\alpha=1}^{m-1} g_{\alpha\alpha,m m}(p),  \\
& & -\sum_{\alpha=1}^{m-1} g_{\alpha\alpha,m m}(p) ~ = ~
\tau_{M}(p) - \tau_{{\mathcal N}}(p) - 2 (m-1)^{2} H_{1}^{2}(p) + 3 (m-1)(m-2)H_{2}(p),    \nonumber \\
& & \int_{{\mathbb R}^{m-1}} |\xi|^{k} \sum_{\alpha , \beta ,\gamma ,\epsilon =1}^{m-1}\, g^{\alpha\beta,\gamma\epsilon} \xi_{\alpha} \xi_{\beta} \xi_{\gamma} \xi_{\epsilon} ~ d \xi ~ = ~ 0 \quad \text{for} \quad k < -3-m .   \nonumber
\end{eqnarray}
\noindent
We summarize these facts as follows (\cite{PS}).
\begin{lemma}   \label{Lemma:3.1}
We consider the boundary normal coordinate system on an open neighborhood $U_{\epsilon_{0}}$ of $p \in {\mathcal N}$ with metric tensor $g = (g_{ij})$
and $p = (0, \cdots, 0)$. Then, we have the following equalities:
\begin{eqnarray*}
& (1) & g^{\alpha\beta,\alpha\beta}(p) = - ~ \frac{1}{3} R_{\alpha\beta\beta\alpha}(p), \qquad
 g^{\alpha\alpha, \beta\beta}(p) = ~ \frac{2}{3} R_{\alpha\beta\beta\alpha}(p),  \\
& (2) & \partial_{x_{\alpha}} \partial_{x_{\alpha}} \ln \big| g \big|(p) = - \frac{2}{3} R_{\alpha\alpha}(p), \qquad
\tau_{{\mathcal N}}(p) = \sum_{\alpha, \beta=1}^{m-1} R_{\alpha\beta\beta\alpha}(p) = - \sum_{\alpha, \gamma=1}^{m-1} R_{\alpha\beta\alpha\beta}(p), \\
& (3) & g^{\alpha\beta,m}(p) = 2 \kappa_{\alpha} \delta_{\alpha\beta} = - g_{\alpha\beta,m}(p),  \qquad
 \int_{{\mathbb R}^{m-1}} |\xi|^{k} \sum_{\alpha ,\beta ,\gamma ,\epsilon =1}^{m-1} g^{\alpha\beta,\gamma\epsilon} \xi_{\alpha} \xi_{\beta} \xi_{\gamma} \xi_{\epsilon} d\xi ~ = ~ 0 \quad \text{for} \quad k < -3-m, \\
& (4) & \sum_{\alpha=1}^{m-1} g^{\alpha\alpha,m m}(p) = 8 \sum_{\alpha=1}^{m-1} \kappa_{\alpha}^{2} -
\sum_{\alpha=1}^{m-1} g_{\alpha\alpha,m m}(p),  \\
& (5) & \sum_{\alpha=1}^{m-1} g_{\alpha\alpha,m m}(p) = - ~ \left\{ \tau_{M}(p) - \tau_{{\mathcal N}}(p) - 2 (m-1)^{2} H_{1}^{2}(p) + 3 (m-1)(m-2) H_{2}(p) \right\},
\end{eqnarray*}
where $\tau_{M}(p)$ and $\tau_{{\mathcal N}}(p)$ are scalar curvatures of $M$ and ${\mathcal N}$ at $p \in {\mathcal N}$, respectively.
\end{lemma}
Using the above equalities and the integral formula (\ref{E:2.28}), we are going to compute the coefficients $\pi_{j}$ and $q_{k}$ in (\ref{E:1.7}).
For this purpose we need to compute the homogeneous symbol of $(\mu - R(\lambda))^{-1}$ by using the metric tensors and Lemma \ref{Lemma:3.1}, and then apply it
to the formula (\ref{E:2.28}).
In fact, this computation is very tedious and complicated. In order to avoid some of the complications and to simplify the presentation, we restrict the argument to the case
of $m = 3$, even though the method presented here works for a manifold of arbitrary dimension.
In the remaining part of this paper we assume that ${\mathcal N}$ is a $2$-dimensional closed Riemannian manifold. In this case (\ref{E:1.4}), (\ref{E:1.6})
and (\ref{E:1.7}) are rewritten as follows,
\begin{eqnarray}   \label{E:3.13}
& & \ln \Det \left( \Delta_{M} + \lambda \right) -
\ln \Det \left( \Delta_{M_{0}, D} + \lambda \right) ~ = ~ a_{0} + a_{1} \lambda +  \ln \Det R(\lambda),  \\
& & \ln \Det \left( \Delta_{M} + \lambda \right) - \ln \Det \left( \Delta_{M_{0}, D} + \lambda \right) ~ \sim ~
- c_{1} \lambda \ln \lambda + c_{3} \ln \lambda + c_{1} \lambda + 2 \sqrt{\pi} c_{2} \lambda^{\frac{1}{2}} - \sqrt{\pi} c_{4} \lambda^{-\frac{1}{2}} + \cdots,  \nonumber  \\
& & \ln \Det R (\lambda) ~ \sim ~ q_{0} \lambda \ln \lambda + q_{1} \lambda^{\frac{1}{2}} \ln \lambda + q_{2} \ln \lambda + \pi_{0} \lambda +
\pi_{1} \lambda^{\frac{1}{2}} +  \pi_{2} +  \pi_{3} \lambda^{-\frac{1}{2}} + \cdots,   \nonumber
\end{eqnarray}
\noindent
where  $a_{j} ~ = ~ \int_{{\mathcal N}} a_{j}(p) ~ d \vol({\mathcal N}) ~$ and
\begin{eqnarray}    \label{E:3.14}
~ \Tr \left( e^{-t \Delta_{M}} - e^{-t \Delta_{M_{0}, D}} \right) \sim \sum_{j=0}^{\infty} c_{j} t^{\frac{-3+j}{2}} \qquad \text{for} \quad t \rightarrow 0^{+},
\end{eqnarray}
\noindent
with $c_{0} = 0$ and $c_{j} = \int_{{\mathcal N}} c_{j}(p) ~ d \vol({\mathcal N})$; see eq.~(\ref{E:1.5}). The density $c_{j}(p)$
for $p \in {\mathcal N}$ is computed in Theorem 3.4.1 of \cite{Gi2} or Section 4.2 of \cite{Ki1} as follows. For $r_{0} = \rank {\mathcal E}$,
\begin{eqnarray}   \label{E:3.15}
c_{0}(p) = 0, \qquad c_{1}(p) = \frac{r_{0}}{8 \pi},  \qquad  c_{2}(p) = 0 .
\end{eqnarray}
\begin{lemma}    \label{Lemma:3.2}
For each $p \in {\mathcal N}$, the density $c_{3}(p)$ is given as
\begin{eqnarray*}
c_{3}(p) & = & r_{0} \left( \frac{\tau_{M}(p)}{64 \pi} + \frac{\tau_{{\mathcal N}}(p)}{192 \pi} - \frac{H_{1}^{2}(p)}{64 \pi} + \frac{H_{2}(p)}{64 \pi} \right)
+ \frac{1}{8 \pi} \Tr E (p).
\end{eqnarray*}
\end{lemma}
\begin{proof}
From Theorem 3.4.1 of \cite{Gi2}, we get
\begin{eqnarray}   \label{E:3.16}
c_{3}(p) & = & (-2) \cdot  r_{0} \cdot \frac{-1}{384} \cdot \frac{1}{4 \pi} \left\{ 16 \tau_{M}(p) + 8 \sum_{\alpha = 1}^{2} R_{\alpha 3 \alpha 3}(p) +
7 (\kappa_{1} + \kappa_{2})^{2}  - 10 (\kappa_{1}^{2} + \kappa_{2}^{2}) \right\} ~ + \frac{1}{8 \pi} \Tr E (p)     \nonumber \\
& = & r_{0} \left\{ \frac{\tau_{M}(p)}{48 \pi} + \frac{1}{96 \pi} \sum_{\alpha = 1}^{2} R_{\alpha 3 \alpha 3}(p) - \frac{H_{1}^{2}(p)}{64 \pi} +
\frac{5 H_{2}(p)}{192 \pi} \right\} + \frac{1}{8 \pi} \Tr E (p) .
\end{eqnarray}
\noindent
Using (\ref{E:2.4}) at a point $p \in {\mathcal N}$, we get
\begin{eqnarray}     \label{E:3.17}
\sum_{\alpha = 1}^{2} R_{\alpha 3 \alpha 3}(p) & = &
\sum_{\alpha =1}^2\, ~ \langle \nabla^{LC}_{\partial_{x_{\alpha}}} \nabla^{LC}_{\partial_{x_{3}}} \partial_{x_{\alpha}} -
\nabla^{LC}_{\partial_{x_{3}}} \nabla^{LC}_{\partial_{x_{\alpha}}} \partial_{x_{\alpha}}, ~ \partial_{x_{3}} \rangle\big|_{p}  \\
& = & \sum_{\alpha=1}^{2} \sum_{l=1}^{3} \langle \nabla^{LC}_{\partial_{x_{\alpha}}} (\Gamma_{3 \alpha}^{l} \partial_{x_{l}}) -
\nabla^{LC}_{\partial_{x_{3}}} (\Gamma_{\alpha \alpha}^{l} \partial_{x_{l}}), ~ \partial_{x_{3}} \rangle\big|_{p}    \nonumber  \\
& = & \sum_{\alpha=1}^{2} \left\{ \Gamma^{3}_{3 \alpha, \alpha}(p) - \Gamma_{\alpha \alpha, 3}^{3}(p) \right\} +
\sum_{\alpha=1}^{2} \sum_{l=1}^{3} \left\{ \Gamma^{l}_{3\alpha}(p) \Gamma^{3}_{\alpha l}(p) - \Gamma^{l}_{\alpha \alpha}(p) \Gamma^{3}_{3 l}(p) \right\}.  \nonumber
\end{eqnarray}
\noindent
By the relation (\ref{E:2.33}), we get the following equalities:
\begin{eqnarray}      \label{E:3.55}
\Gamma^{3}_{3j} (x) & = & \frac{1}{2} \sum_{s=1}^{3} g^{3s} \left( g_{js,3} + g_{s3,j} - g_{3j,s} \right) ~ = ~ \frac{1}{2} \left( g_{j3,3} + g_{33,j} - g_{3j,3} \right)
~ = ~ 0,  \\
\Gamma^{3}_{\alpha \alpha}(x) & = & \frac{1}{2} \sum_{s=1}^{3} g^{3 s} \left( g_{\alpha s, \alpha} + g_{s \alpha, \alpha} - g_{\alpha \alpha, s} \right)
~ = ~ \frac{1}{2} \left( g_{\alpha 3, \alpha} + g_{3 \alpha, \alpha} - g_{\alpha \alpha, 3} \right) ~ = ~ - \frac{1}{2} g_{\alpha \alpha, 3}(x),    \nonumber   \\
\Gamma^{l}_{3 \alpha}(p) & = &  \frac{1}{2} \sum_{s=1}^{3} g^{ls} \left( g_{\alpha s, 3} + g_{s 3, \alpha} - g_{3 \alpha,s} \right)\big|_{p}
~ = ~ \frac{1}{2} \left( g_{\alpha l, 3} + g_{l 3, \alpha} - g_{3 \alpha, l} \right)\big|_{p}
~ = ~ \frac{1}{2} g_{\alpha l, 3}(p) ~ = ~ - \kappa_{\alpha} \delta_{\alpha l},     \nonumber  \\
\Gamma^{3}_{\alpha l}(p) & = &  \frac{1}{2} \sum_{s=1}^{3} g^{3 s} \left( g_{l s, \alpha} + g_{s \alpha, l} - g_{\alpha l,s} \right)\big|_{p}
~ = ~ \frac{1}{2} \left( g_{l 3, \alpha} + g_{3 \alpha, l} - g_{\alpha l, 3} \right)\big|_{p}
~ = ~ - \frac{1}{2} g_{\alpha l, 3}(p) ~ = ~ \kappa_{\alpha} \delta_{\alpha l},    \nonumber
\end{eqnarray}
\noindent
which shows that
\begin{eqnarray*}
\sum_{\alpha = 1}^{2} R_{\alpha 3 \alpha 3}(p) & = & - \sum_{\alpha = 1}^{2} \Gamma^{3}_{\alpha \alpha, 3}(p) - (\kappa_{1}^{2} + \kappa_{2}^{2})
~ = ~  \frac{1}{2} \sum_{\alpha = 1}^{2} g_{\alpha \alpha, 3 3}(p) - 4 H_{1}^{2}(p) + 2 H_{2}(p)  \\
& = & - ~ \frac{1}{2} \left( \tau_{M}(p) - \tau_{{\mathcal N}}(p) \right) - H_{2}(p),
\end{eqnarray*}
\noindent
where in the last equality we used the statement (5) in Lemma \ref{Lemma:3.1}.
This together with (\ref{E:3.16}) yields the result.
\end{proof}
We now compute the densities $\pi_{j}(p)$ and $q_{j}(p)$ for $p \in {\mathcal N}$  by using (\ref{E:2.27}) and (\ref{E:2.28}).
The following lemma is straightforward and will turn out to be very helpful.
\begin{lemma}   \label{Lemma:3.33}
Let ${\mathbb C}^+ = {\mathbb C} -\{r\in {\mathbb R} | r\leq 0\}$. For $z\in {\mathbb C}^+$ let $\gamma$ be a counterclockwise contour in ${\mathbb C}^+$ with
$z$ inside $\gamma$.
Then for $\re s > 2$ the following integrals are all well defined and one computes:
\begin{eqnarray*}
&(1)& \frac{1}{2 \pi i} \int_{\gamma} \frac{\mu^{-s}}{\mu - z} d\mu = z^{-s}, \quad
\frac{1}{2 \pi i} \int_{\gamma} \frac{\mu^{-s}}{(\mu - z)^{2}} d\mu = - s z^{-s-1}, \quad
\frac{1}{2 \pi i} \int_{\gamma} \frac{\mu^{-s}}{(\mu - z)^{3}} d\mu = \frac{1}{2} s (s + 1) z^{-s-2},  \\
&(2)& \frac{1}{4 \pi^{2}} \int_{{\mathbb R}^{2}} ( |\xi|^{2} + 1)^{-\frac{s}{2}} d \xi ~ = ~ \frac{1}{2 \pi} \frac{1}{s-2},  \qquad
 \frac{1}{4 \pi^{2}} \int_{{\mathbb R}^{2}} ( |\xi|^{2} + 1)^{-\frac{s}{2} -1} d \xi ~ = ~ \frac{1}{2 \pi} \frac{1}{s},  \\
& & \frac{1}{4 \pi^{2}} \int_{{\mathbb R}^{2}} \xi_{1}^{2} ( |\xi|^{2} + 1)^{-\frac{s}{2} - 2} d \xi ~ = ~ \frac{1}{2 \pi} \frac{1}{s (s + 2)},  \\
& & \frac{1}{4 \pi^{2}} \int_{{\mathbb R}^{2}} \xi_{1}^{2} \xi_{2}^{2} ( |\xi|^{2} + 1)^{-\frac{s}{2} - 3} d \xi ~ = ~ \frac{1}{2 \pi} \frac{1}{s (s + 2) (s + 4)},  \\
& & \frac{1}{4 \pi^{2}} \int_{{\mathbb R}^{2}} \xi_{1}^{4} ( |\xi|^{2} + 1)^{-\frac{s}{2} - 3} d \xi ~ = ~ \frac{3}{2 \pi} \frac{1}{s (s + 2)(s + 4)}.
\end{eqnarray*}
\end{lemma}
\subsection{The coefficients $\pi_{0}$ and $q_{0}$}
From (\ref{E:2.27}) and (\ref{E:2.28}), we note that
\begin{eqnarray*}
r_{-1} \left(x, \xi, \frac{\lambda}{|\lambda|}, \mu \right)|_{x=p} & = & (\mu - 2 \sqrt{|\xi|^{2} + 1})^{-1}
\end{eqnarray*}
\noindent
which shows that
\begin{eqnarray*}
\frac{1}{(2 \pi)^{2}} \int_{T_{p}^{\ast}{\mathcal N}}
\frac{1}{2 \pi i} \int_{\gamma} \mu^{-s} \Tr r_{-1} \left( x, \xi, \frac{\lambda}{|\lambda|}, \mu \right)|_{x=p}
\,\, d\mu d \xi & = &
\frac{r_{0}}{(2 \pi)^{2}} \int_{T_{p}^{\ast}{\mathcal N}} 2^{-s} (|\xi|^{2} + 1)^{-\frac{s}{2}} d \xi  \\
& = & \frac{r_{0}}{2 \pi} \cdot \frac{2^{-s}}{s-2}.
\end{eqnarray*}
\noindent
Hence,
\begin{eqnarray}     \label{E:3.19}
\pi_{0}(p) ~ = ~ - r_{0} \left( \frac{1}{4 \pi} \ln 2 - \frac{1}{8 \pi} \right), \qquad
q_{0}(p) ~ = ~ - \frac{r_{0}}{8 \pi}.
\end{eqnarray}
\noindent
Denoting the volume form on ${\mathcal N}$ as before by $d \vol({\mathcal N})$, these results yield the following equalities by integration,
\begin{eqnarray}     \label{E:3.20}
\pi_{0} & = & \int_{{\mathcal N}} \pi_{0}(p) ~ d \vol({\mathcal N}) ~ = ~
r_{0} \left( - \frac{1}{4 \pi} \ln 2 + \frac{1}{8 \pi} \right) \vol({\mathcal N}),  \\
q_{0} & = & \int_{{\mathcal N}} q_{0}(p) ~ d \vol({\mathcal N}) ~ = ~ - \frac{r_{0}}{8 \pi} \vol({\mathcal N}).   \nonumber
\end{eqnarray}
\noindent
{\it Remark} : If we compare the coefficients of $\lambda \ln \lambda$ in the two asymptotic expansions in (\ref{E:3.13}), we see that $q_{0} = - c_{1}$,
which together with (\ref{E:3.15}) implies the second equality of (\ref{E:3.20}).
\subsection{The coefficients $\pi_{1}$ and $q_{1}$}
Since $r_{-2} \left( x, \xi, \frac{\lambda}{|\lambda|}, \mu \right)$ is an odd function with respect to $\xi$,
the densities vanish, {\it i.e.} $\pi_{1}(p) = q_{1}(p) = 0$, and hence  $\pi_{1} = q_{1} = 0$.
\subsection{The coefficients $\pi_{2}$ and $q_{2}$}
We recall from (\ref{E:2.27}) that
\begin{eqnarray*}
& & r_{-3} (x, \xi, \lambda, \mu)|_{x=p}   \\
& = & \left( \mu - 2 \sqrt{|\xi|^{2} + \lambda} \right)^{-1} \Big\{ \sum_{|\omega |=2} \frac{1}{\omega !} \partial_{\xi}^{\omega} \theta_{1} \cdot D_{x}^{\omega} r_{-1}
+ \partial_{\xi} \theta_{1} \cdot D_{x} r_{-2} \\
& &\hspace{5.0cm}+ \partial_{\xi} \theta_{0} \cdot D_{x} r_{-1} + \theta_{0} \cdot r_{-2}
 + \theta_{-1} \cdot r_{-1} \Big\}\big|_{x = p}  \\
& =: & (\Aa) + (\Bb) + (\Cc) + (\Dd) + (\Ee).
\end{eqnarray*}
\noindent
For two integrable functions $f(\xi)$ and $g(\xi)$ on ${\mathbb R}^{2}$, we define an equivalence relation "$~ \approx ~$" as follows:
\begin{eqnarray*}
f & \approx & g  \qquad \text{if and only if} \qquad \int_{{\mathbb R}^{2}} f(\xi) ~ d \xi ~ = ~ \int_{{\mathbb R}^{2}} g(\xi) ~ d \xi.
\end{eqnarray*}
\noindent
In the remaining part of this section we perform a complicated and tedious computation in order to find $(\Aa ) - (\Ee )$.
We keep in mind that we are using the boundary normal coordinates on an open neighborhood of $p \in {\mathcal N}$.
\begin{eqnarray*}
(\Aa) & = & (\mu - 2 \sqrt{|\xi|^{2} + \lambda})^{-1} \left( \sum_{|\omega |=2} \frac{1}{\omega !} \partial_{\xi}^{\omega} \theta_{1} \cdot D_{x}^{\omega} r_{-1} \right)\bigg|_{x = p}  \\
& = & (\mu - 2 \sqrt{|\xi|^{2} + \lambda})^{-1} \left\{ \frac{1}{2}
\sum_{\alpha =1}^{2}\, \partial_{\xi_{\alpha}} \partial_{\xi_{\alpha}} \left( 2 \sqrt{|\xi|^{2} + \lambda} \right) \cdot (-1)
\partial_{x_{\alpha}} \partial_{x_{\alpha}} \left(\mu - 2 \sqrt{|\xi|^{2} + \lambda} \right)^{-1}  \right.  \\
& & \left. + \partial_{\xi_{1}} \partial_{\xi_{2}} \left( 2 \sqrt{|\xi|^{2} + \lambda} \right) \cdot (-1)
\partial_{x_{1}} \partial_{x_{2}} \left(\mu - 2 \sqrt{|\xi|^{2} + \lambda} \right)^{-1} \right\}\bigg|_{x = p} . \\
\end{eqnarray*}
\noindent
It is straightforward that
\begin{eqnarray*}
& & - \frac{1}{2} \partial_{\xi_{\alpha}} \partial_{\xi_{\alpha}} \left( 2 \sqrt{|\xi|^{2} + \lambda} \right)\bigg|_{x=p} ~ = ~
\frac{- g^{\alpha \alpha}}{\sqrt{|\xi|^{2} + \lambda}} + \frac{(\sum_{\gamma =1}^2\,g^{\alpha\gamma} \xi_{\gamma})(\sum_{\kappa =1}^2\,g^{\alpha\kappa} \xi_{\kappa})}
{(|\xi|^{2} + \lambda)^{3/2}} \bigg|_{x=p},   \\
& & \partial_{x_{\alpha}} \partial_{x_{\alpha}} \left(\mu - 2 \sqrt{|\xi|^{2} + \lambda} \right)^{-1}\big|_{x = p} ~ = ~ \frac{\sum_{\mu ,\nu =1}^2\,
g^{\mu\nu, \alpha\alpha} \xi_{\mu}\xi_{\nu}}{\left(\mu - 2 \sqrt{|\xi|^{2} + \lambda} \right)^{2} \sqrt{|\xi|^{2} + \lambda}} \bigg|_{x=p},
\end{eqnarray*}
which shows that
\begin{eqnarray*}
& & - \frac{1}{2} \sum_{\alpha =1}^2\, \partial_{\xi_{\alpha}} \partial_{\xi_{\alpha}} \left( 2 \sqrt{|\xi|^{2} + \lambda} \right) \cdot
\partial_{x_{\alpha}} \partial_{x_{\alpha}} \left(\mu - 2 \sqrt{|\xi|^{2} + \lambda} \right)^{-1}\big|_{x = p}  \\
& \approx & \frac{- \sum_{\alpha ,\kappa =1}^2\, g^{\kappa\kappa,\alpha\alpha} \xi_{\kappa}^{2}}{\left(\mu - 2 \sqrt{|\xi|^{2} + \lambda} \right)^{2} (|\xi|^{2} + \lambda)} +
\frac{\sum_{\alpha ,\theta =1}^2 \,g^{\theta\theta,\alpha\alpha} \xi_{\alpha}^{2} \xi_{\theta}^{2}}{\left(\mu - 2 \sqrt{|\xi|^{2} + \lambda} \right)^{2} (|\xi|^{2} + \lambda)^{2}}  \bigg|_{x=p}\\
& \approx & \frac{- \frac{2}{3} \tau_{{\mathcal N}}(p) \xi_{1}^{2}}{\left(\mu - 2 \sqrt{|\xi|^{2} + \lambda} \right)^{2} (|\xi|^{2} + \lambda)} +
\frac{\frac{2}{3} \tau_{{\mathcal N}}(p) \xi_{1}^{2} \xi_{2}^{2}}{\left(\mu - 2 \sqrt{|\xi|^{2} + \lambda} \right)^{2} (|\xi|^{2} + \lambda)^{2}}.
\end{eqnarray*}
\noindent
Similarly,
\begin{eqnarray*}
& & - ~ \partial_{\xi_{1}} \partial_{\xi_{2}} \left( 2 \sqrt{|\xi|^{2} + \lambda} \right) \cdot
\partial_{x_{1}} \partial_{x_{2}} \left(\mu - 2 \sqrt{|\xi|^{2} + \lambda} \right)^{-1}  \big|_{x=p}\\
& \approx & \frac{4 g^{12,12} \xi_{1}^{2} \xi_{2}^{2}}{\left(\mu - 2 \sqrt{|\xi|^{2} + \lambda} \right)^{2} (|\xi|^{2} + \lambda)^{2}} \big|_{x=p} ~ = ~
\frac{- \frac{4}{3} R_{1221} \xi_{1}^{2} \xi_{2}^{2}}{\left(\mu - 2 \sqrt{|\xi|^{2} + \lambda} \right)^{2} (|\xi|^{2} + \lambda)^{2}} \big|_{x=p} \\
& = & \frac{- \frac{2}{3} \tau_{{\mathcal N}}(p) \xi_{1}^{2} \xi_{2}^{2}}{\left(\mu - 2 \sqrt{|\xi|^{2} + \lambda} \right)^{2} (|\xi|^{2} + \lambda)^{2}}.
\end{eqnarray*}
\noindent
Adding the above two terms, we get
\begin{eqnarray*}
(\Aa) & \approx &  \frac{- \frac{2}{3} ~ \tau_{{\mathcal N}}(p) ~ \xi_{1}^{2}}{(\mu - 2 \sqrt{|\xi|^{2} + \lambda})^{3}(|\xi|^{2} + \lambda)} ,
\end{eqnarray*}
\noindent
which leads to
\begin{eqnarray}    \label{E:3.22}
\frac{1}{(2 \pi)^{2}}\int_{T_{p}^{\ast}{\mathcal N}}
\frac{1}{2 \pi i} \int_{\gamma} \mu^{-s} ~ \Tr (\Aa)\left( x, \xi, \frac{\lambda}{|\lambda|}, \mu \right) ~ d\mu ~ d \xi & = &
- ~ r_{0} \cdot \frac{\tau_{{\mathcal N}}(p)}{24 \pi} \cdot 2^{-s} \cdot \frac{s+1}{s+2}.
\end{eqnarray}
\noindent
The other terms are computed in the same way. Each term is evaluated at $x = p \in {\mathcal N}$,
which we sometimes omit for simple presentation. The results of the other cases are given as follows.
\begin{eqnarray*}
(\Bb) & = & \frac{\partial_{\xi} \theta_{1} \cdot D_{x} r_{-2}}{\mu - 2 \sqrt{|\xi|^{2} + \lambda}}\bigg|_{x = p}  \\
& = & \frac{(-i) \sum_{\alpha =1}^2\, \left( \partial_{\xi_{\alpha}} 2 \sqrt{|\xi|^{2} + \lambda}\right) \cdot \partial_{x_{\alpha}}
\left( ( \mu - 2 \sqrt{|\xi|^{2} + \lambda})^{-1} ( \sum_{\beta =1}^2\, \partial_{\xi_{\beta}} \theta_{1} \cdot D_{x_{\beta}} r_{-1} + \theta_{0} \cdot r_{-1}) \right)}{\mu - 2 \sqrt{|\xi|^{2} + \lambda}}  \\
& \approx & \frac{- 2 \left( \sum_{\alpha ,\gamma =1}^2\, g^{\gamma\alpha,\gamma\alpha} + \frac{1}{2} \sum_{\alpha =1}^2\, \partial_{x_{\alpha}} \partial_{x_{\alpha}} \ln \big| g \big|(p) \right) \xi_{1}^{2}}{(\mu - 2 \sqrt{|\xi|^{2} + \lambda})^{3} (|\xi|^{2} + \lambda)}
~ + ~ \frac{-4 \sum_{\alpha =1}^2\, \left( \partial_{x_{\alpha}} \omega_{\alpha} \right) \xi_{\alpha}^{2}}{(\mu - 2 \sqrt{|\xi|^{2} + \lambda})^{3} (|\xi|^{2} + \lambda)}\\
& = & \frac{- 2 \left( - \sum_{\alpha ,\gamma =1}^2\, \frac{1}{3} R_{\gamma\alpha\alpha\gamma} + \frac{1}{2} \sum_{\alpha =1}^2\, (-\frac{2}{3} R_{\alpha\alpha}) \right) \xi_{1}^{2}}
{(\mu - 2 \sqrt{|\xi|^{2} + \lambda})^{3} (|\xi|^{2} + \lambda)}
~ + ~ \frac{-4 \sum_{\alpha =1}^2\, \left( \partial_{x_{\alpha}} \omega_{\alpha} \right) \xi_{\alpha}^{2}}{(\mu - 2 \sqrt{|\xi|^{2} + \lambda})^{3} (|\xi|^{2} + \lambda)} \\
& \approx & \frac{\frac{4}{3} \tau_{{\mathcal N}}(p) \xi_{1}^{2} -4 \sum_{\alpha =1}^2 \left( \partial_{x_{\alpha}} \omega_{\alpha} \right) \xi_{1}^{2}}{(\mu - 2 \sqrt{|\xi|^{2} + \lambda})^{3} (|\xi|^{2} + \lambda)},
\end{eqnarray*}
\noindent
which leads to
\begin{eqnarray}   \label{E:3.23}
& & \frac{1}{(2 \pi)^{2}}\int_{T_{p}^{\ast}{\mathcal N}}
 \frac{1}{2 \pi i} \int_{\gamma} \mu^{-s} ~ \Tr (\Bb)\left(x, \xi, \frac{\lambda}{|\lambda|}, \mu \right) ~ d\mu ~ d \xi    \nonumber  \\
& = & r_{0} \cdot \frac{\tau_{{\mathcal N}}(p)}{12 \pi} \cdot 2^{-s} \cdot \frac{s+1}{s+2}
~ - ~ \frac{1}{4 \pi} \left( \Tr \sum_{\alpha} \partial_{x_{\alpha}} \omega_{\alpha}\big|_{x = p} \right)
\cdot 2^{-s} \cdot \frac{s+1}{s+2}.
\end{eqnarray}
\begin{eqnarray}   \label{E:3.24}
(\Cc) & = & \frac{\sum_{\alpha =1}^2\, \partial_{\xi_{\alpha}} \theta_{0} \cdot D_{x_{\alpha}} r_{-1}}{\mu - 2 \sqrt{|\xi|^{2} + \lambda}}\bigg|_{x = p} ~ = ~
\frac{(-i) \sum_{\alpha =1}^2\, \partial_{\xi_{\alpha}} \theta_{0} \cdot \partial_{x_{\alpha}}
(\mu - 2 \sqrt{|\xi|^{2} + \lambda})^{-1}}{\mu - 2 \sqrt{|\xi|^{2} + \lambda}}\bigg|_{x = p}      \nonumber \\
& = & \frac{(-i) \sum_{\alpha =1}^2\, \partial_{\xi_{\alpha}} \theta_{0} \cdot \sum_{\gamma ,\delta =1}^2\, g^{\gamma \delta, \alpha} \xi_{\gamma} \xi_{\delta}}{(\mu - 2 \sqrt{|\xi|^{2} + \lambda})^{3} \sqrt{|\xi|^{2} + \lambda}}\bigg|_{x = p} ~ = ~ 0,
\end{eqnarray}
\noindent
since $g^{\gamma \delta, \alpha}(p) = 0$.
\begin{eqnarray*}
(\Dd) & = & \frac{\theta_{0} \cdot r_{-2}}{\mu - 2 \sqrt{|\xi|^{2} + \lambda}}\big|_{x = p} ~ = ~
\frac{\theta_{0}^{2}}{(\mu - 2 \sqrt{|\xi|^{2} + \lambda})^{3}}\big|_{x = p} ~ \approx ~
\frac{\left(-4 \sum_{\alpha =1}^2\, \omega_{\alpha} \omega_{\alpha} \right) \xi_{1}^{2}}{(\mu - 2 \sqrt{|\xi|^{2} + \lambda})^{3} (|\xi|^{2} + \lambda)},
\end{eqnarray*}
which leads to
\begin{eqnarray}  \label{E:3.26}
\frac{1}{(2 \pi)^{2}}\int_{T_{p}^{\ast}{\mathcal N}}
\frac{1}{2 \pi i} \int_{\gamma} \mu^{-s} ~ \Tr (\Dd)\left(x, \xi, \frac{\lambda}{|\lambda|}, \mu\right) ~ d\mu ~ d \xi  =
- \frac{1}{4 \pi} \left( \Tr \sum_{\alpha =1}^2\, \omega_{\alpha} \omega_{\alpha}\big|_{x = p} \right) \cdot 2^{-s} \cdot \frac{s+1}{s+2}.
\end{eqnarray}
\begin{eqnarray*}
& & (\Ee) ~ = ~ \frac{\theta_{-1} \cdot r_{-1}}{\mu - 2 \sqrt{|\xi|^{2} + \lambda}}\big|_{x = p} ~ = ~
\frac{\theta_{-1}}{(\mu - 2 \sqrt{|\xi|^{2} + \lambda})^{2}}\big|_{x = p} ~ = ~
\frac{1}{2(\mu - 2 \sqrt{|\xi|^{2} + \lambda})^{2} \sqrt{|\xi|^{2} + \lambda}} \times  \\
& & \Big\{ - \sum_{| \omega| =2} \frac{2}{\omega !} \partial_{\xi}^{\omega} \alpha_{1} \cdot D_{x}^{\omega} \alpha_{1}
 - \partial_{\xi} \theta_{0} \cdot D_{x} \alpha_{1} - \partial_{\xi} \alpha_{1} \cdot D_{x} \theta_{0}
- (\alpha_{0}^{2} + \beta_{0}^{2})  \\
& & \hspace{1.0cm}+ A(x,0) (\alpha_{0} - \beta_{0}) - \partial_{x_{m}} (\alpha_{0} - \beta_{0}) + 2 p_{0} \Big\}\big|_{x = p}  \\
& & =: ~ (1) + (2) + (3) + (4) + (5) + (6) + (7).
\end{eqnarray*}
\begin{eqnarray*}
(1) & = & \frac{- \sum_{| \omega | =2} \frac{2}{\omega !} \partial_{\xi}^{\omega} \alpha_{1} \cdot D_{x}^{\omega} \alpha_{1}}{2(\mu - 2 \sqrt{|\xi|^{2} + \lambda})^{2} \sqrt{|\xi|^{2} + \lambda}}\big|_{x = p}
~ \approx ~ \frac{\tau_{{\mathcal N}}(p)}{6} \cdot \frac{\xi_{1}^{2}}{(\mu - 2 \sqrt{|\xi|^{2} + \lambda})^{2} ( |\xi|^{2} + \lambda)^{3/2}} ,
\end{eqnarray*}
\noindent
which leads to
\begin{eqnarray*}
\frac{1}{(2 \pi)^{2}}\int_{T_{p}^{\ast}{\mathcal N}}
\frac{1}{2 \pi i} \int_{\gamma} \mu^{-s} ~ \Tr (1)\left(x, \xi, \frac{\lambda}{|\lambda|}, \mu\right) ~ d\mu ~ d \xi & = &
- r_{0} \cdot \frac{\tau_{{\mathcal N}}(p)}{24 \pi} \cdot 2^{-s} \cdot \frac{1}{s+2}.
\end{eqnarray*}
\begin{eqnarray*}
(2) & = & \frac{- \partial_{\xi} \theta_{0} \cdot D_{x} \alpha_{1}}
{2(\mu - 2 \sqrt{|\xi|^{2} + \lambda})^{2} \sqrt{|\xi|^{2} + \lambda}}\big|_{x = p}
~ = ~ \frac{i \sum_{\alpha =1}^2\, \partial_{\xi_{\alpha}} \theta_{0} \cdot \sum_{\beta ,\gamma =1}^2\,
g^{\beta\gamma,\alpha} \xi_{\beta} \xi_{\gamma}}{4(\mu - 2 \sqrt{|\xi|^{2} + \lambda})^{2} (|\xi|^{2} + \lambda)}\big|_{x = p}
~ = ~ 0,
\end{eqnarray*}
\noindent
since $g^{\beta\gamma,\alpha}(p) = 0$.
\begin{eqnarray*}
(3) & = & \frac{- \sum_{\alpha =1}^2\, \partial_{\xi_{\alpha}} \alpha_{1} \cdot D_{x_{\alpha}} \theta_{0}}{2(\mu - 2 \sqrt{|\xi|^{2} + \lambda})^{2} \sqrt{|\xi|^{2} + \lambda}}\bigg|_{x = p}  \\
 &\approx & \left\{ \frac{\sum_{\alpha ,\gamma =1}^2\, g^{\gamma\alpha,\gamma\alpha} \xi_{\alpha}^{2} + \frac{1}{2} \sum_{\alpha =1}^2\,
 \partial_{\alpha} \partial_{\alpha} \log | g |(x) \xi_{\alpha}^{2}}{2(\mu - 2 \sqrt{|\xi|^{2} + \lambda})^{2} \left(|\xi|^{2} + \lambda\right)^{3/2}}  - \frac{\sum_{\kappa ,\theta ,\gamma ,\alpha =1}^2\, g^{\kappa\theta,\gamma\alpha}\xi_{\alpha} \xi_{\gamma} \xi_{\kappa} \xi_{\theta}}{4(\mu - 2 \sqrt{|\xi|^{2} + \lambda})^{2} \left(|\xi|^{2} + \lambda\right)^{5/2}}  \right\} \Id\\
 & & + ~
\frac{\sum_{\alpha =1}^2\, \left( \partial_{x_{\alpha}} \omega_{\alpha} \right) \xi_{\alpha}^{2}}{(\mu - 2 \sqrt{|\xi|^{2} + \lambda})^{2} \left(|\xi|^{2} + \lambda\right)^{3/2}}      \\
& \approx & \frac{\sum_{\alpha ,\gamma =1}^2\, (- \frac{1}{3}) R_{\gamma\alpha\alpha\gamma} \xi_{\alpha}^{2} + \frac{1}{2} \sum_{\alpha =1}^2\, (-\frac{2}{3}) R_{\alpha\alpha}  \xi_{\alpha}^{2}}{2(\mu - 2 \sqrt{|\xi|^{2} + \lambda})^{2} \left(|\xi|^{2} + \lambda\right)^{3/2}} \Id  ~ + ~
\frac{\sum_{\alpha =1}^2\, \left( \partial_{x_{\alpha}} \omega_{\alpha} \right) \xi_{\alpha}^{2}}{(\mu - 2 \sqrt{|\xi|^{2} + \lambda})^{2} \left(|\xi|^{2} + \lambda\right)^{3/2}}  \\
& \approx & \frac{- \frac{1}{3} \tau_{{\mathcal N}}(p)  \xi_{1}^{2}}{(\mu - 2 \sqrt{|\xi|^{2} + \lambda})^{2} \left(|\xi|^{2} + \lambda\right)^{3/2}} \Id
+ \frac{\sum_{\alpha =1}^2\, \left( \partial_{x_{\alpha}} \omega_{\alpha} \right) \xi_{1}^{2}}{(\mu - 2 \sqrt{|\xi|^{2} + \lambda})^{2} \left(|\xi|^{2} + \lambda\right)^{3/2}}\big|_{x = p},
\end{eqnarray*}
\noindent
which leads to
\begin{eqnarray*}
& & \frac{1}{(2 \pi)^{2}}\int_{T_{p}^{\ast}{\mathcal N}}
\frac{1}{2 \pi i} \int_{\gamma} \mu^{-s} ~ \Tr (3)\left(x, \xi, \frac{\lambda}{|\lambda|}, \mu\right) ~ d\mu ~ d \xi \\
& = & r_{0} \cdot \frac{\tau_{{\mathcal N}}(p)}{12 \pi} \cdot 2^{-s} \cdot \frac{1}{s+2}
- \frac{1}{4 \pi} \left( \sum_{\alpha =1}^2\, \Tr \partial_{x_{\alpha}} \omega_{\alpha}\big|_{x = p} \right)
\cdot 2^{-s} \cdot \frac{1}{s+2} ~ .
\end{eqnarray*}
\begin{eqnarray*}
(4) & = & \frac{- (\alpha_{0}^{2} + \beta_{0}^{2})}{2(\mu - 2 \sqrt{|\xi|^{2} + \lambda})^{2} \sqrt{|\xi|^{2} + \lambda}}\bigg|_{x = p}  \\
& = & \frac{\left( \sum_{\alpha =1}^2\, \omega_{\alpha} \xi_{\alpha} \right)^{2}}{(\mu - 2 \sqrt{|\xi|^{2} + \lambda})^{2} \left(|\xi|^{2} + \lambda\right)^{3/2}}
~ + ~ \frac{- \left( \frac{1}{2} A(x,0)^{2} - \frac{A(x,0)}{\alpha_{1}} \partial_{x_{m}} \alpha_{1} +
\frac{1}{2 \alpha_{1}^{2}} \left( \partial_{x_{m}} \alpha_{1} \right)^{2}\right)}{2(\mu - 2 \sqrt{|\xi|^{2} + \lambda})^{2} \sqrt{|\xi|^{2} + \lambda}} \\
& \approx & \frac{- H_{1}^{2}}{(\mu - 2 \sqrt{|\xi|^{2} + \lambda})^{2} \sqrt{|\xi|^{2} + \lambda}}
~ + ~ \frac{(2 H_{1}^{2}) \xi_{1}^{2} + (\sum_{\alpha =1}^2\, \omega_{\alpha} \omega_{\alpha}) \xi_{1}^{2}}{(\mu - 2 \sqrt{|\xi|^{2} + \lambda})^{2} \left(|\xi|^{2} + \lambda\right)^{3/2}}    \\
& & - ~ \frac{(4 H_{1}^{2} - 2 H_{2}) \xi_{1}^{4} + 2 H_{2} \xi_{1}^{2} \xi_{2}^{2}}{4 (\mu - 2 \sqrt{|\xi|^{2} + \lambda})^{2} \left(|\xi|^{2} + \lambda\right)^{5/2}} ,
\end{eqnarray*}
\noindent
which leads to
\begin{eqnarray*}
& & \frac{1}{(2 \pi)^{2}}\int_{T_{p}^{\ast}{\mathcal N}}
\frac{1}{2 \pi i} \int_{\gamma} \mu^{-s} ~ \Tr (4)\left(x, \xi, \frac{\lambda}{|\lambda|}, \mu\right) ~ d\mu ~ d \xi ~ = ~
 - \frac{1}{4 \pi} \sum_{\alpha =1}^2\, \Tr \left( \omega_{\alpha} \omega_{\alpha}\big|_{x=p} \right) \cdot 2^{-s} \cdot \frac{1}{s+2}  \\
&  & \hspace{1.5 cm}  + ~ r_{0} \left\{ \frac{H_{1}^{2}(p)}{4 \pi} \cdot 2^{-s} \left( \frac{s}{s+2} + \frac{3}{(s+2)(s+4)} \right) - \frac{H_{2}(p)}{4 \pi} \cdot 2^{-s} \frac{1}{(s+2)(s+4)} \right\} .
\end{eqnarray*}
\begin{eqnarray*}
(5) & = & \frac{A(x,0) (\alpha_{0} - \beta_{0})}{2(\mu - 2 \sqrt{|\xi|^{2} + \lambda})^{2} \sqrt{|\xi|^{2} + \lambda}}\bigg|_{x=p}
~ = ~
\frac{A(x,0)^{2} - \frac{A(x,0)}{\alpha_{1}} \partial_{x_{m}} \alpha_{1}}{2(\mu - 2 \sqrt{|\xi|^{2} + \lambda})^{2} \sqrt{|\xi|^{2} + \lambda}}\bigg|_{x=p}  \\
& \approx & \frac{ 2 H_{1}^{2}}{(\mu - 2 \sqrt{|\xi|^{2} + \lambda})^{2} \sqrt{|\xi|^{2} + \lambda}}  +
\frac{\left( -  2 H_{1}^{2} \right) \xi_{1}^{2}}{(\mu - 2 \sqrt{|\xi|^{2} + \lambda})^{2} \left(|\xi|^{2} + \lambda\right)^{3/2}} ,
\end{eqnarray*}
\noindent
which leads to
\begin{eqnarray*}
 \frac{1}{(2 \pi)^{2}}\int_{T_{p}^{\ast}{\mathcal N}}
\frac{1}{2 \pi i} \int_{\gamma} \mu^{-s} ~ \Tr (5)\left(x, \xi, \frac{\lambda}{|\lambda|}, \mu\right) ~ d\mu ~ d \xi
& = & - ~ r_{0} \cdot \frac{H_{1}^{2}(p)}{2 \pi} \cdot 2^{-s} \cdot \frac{s+1}{s+2}.
\end{eqnarray*}
\begin{eqnarray*}
(6) & = & \frac{- \partial_{x_{m}} (\alpha_{0} - \beta_{0})}{2(\mu - 2 \sqrt{|\xi|^{2} + \lambda})^{2} \sqrt{|\xi|^{2} + \lambda}}\bigg|_{x=p} ~ = ~
\frac{- \partial_{x_{m}} \left( A(x,0) - \frac{1}{\alpha_{1}} \partial_{x_{m}} \alpha_{1} \right)}{2(\mu - 2 \sqrt{|\xi|^{2} + \lambda})^{2} \sqrt{|\xi|^{2} + \lambda}}\bigg|_{x=p}  \\
& \approx & \left\{ \frac{-4H_{1}^{2} + H_{2} - \frac{1}{2} (\tau_{M}(x) - \tau_{{\mathcal N}}(x))}{2(\mu - 2 \sqrt{|\xi|^{2} + \lambda})^{2} \sqrt{|\xi|^{2} + \lambda}}
~ + ~ \frac{\left( 12 H_{1}^{2} - 5 H_{2} + \frac{1}{2} (\tau_{M}(x) - \tau_{{\mathcal N}}(x)) \right) \xi_{1}^{2}}{2(\mu - 2 \sqrt{|\xi|^{2} + \lambda})^{2} \left(|\xi|^{2} + \lambda\right)^{3/2}} \right. \\
& & \left. + ~ \frac{(-4H_{1}^{2} + 2H_{2}) \xi_{1}^{4} - 2 H_{2} \xi_{1}^{2} \xi_{2}^{2}}{(\mu - 2 \sqrt{|\xi|^{2} + \lambda})^{2} \left(|\xi|^{2} + \lambda\right)^{5/2}} \right\}\bigg|_{x=p} ,
\end{eqnarray*}
\noindent
which leads to
\begin{eqnarray*}
& & \frac{1}{(2 \pi)^{2}}\int_{T_{p}^{\ast}{\mathcal N}}
\frac{1}{2 \pi i} \int_{\gamma} \mu^{-s} ~ \Tr (6)\left(x, \xi, \frac{\lambda}{|\lambda|}, \mu\right) ~ d\mu ~ d \xi  \\
& = & r_{0} \left\{ \frac{H_{1}^{2}(p)}{2 \pi} \cdot 2^{-s} \cdot \frac{s+1}{s+4} ~ - ~ \frac{H_{2}(p)}{8 \pi} \cdot 2^{-s} \cdot \frac{s^{2} + s - 4}{(s+2)(s+4)} ~
 + ~ \frac{1}{16 \pi} \left( \tau_{M}(p) - \tau_{{\mathcal N}}(p) \right) \cdot 2^{-s} \cdot \frac{s+1}{s+2} \right\}.
\end{eqnarray*}
\begin{eqnarray*}
(7) & = & \frac{2p_{0}}{2(\mu - 2 \sqrt{|\xi|^{2} + \lambda})^{2} \sqrt{|\xi|^{2} + \lambda}}\bigg|_{x=p} ~ = ~
- \frac{\sum_{\alpha , \beta =1}^2\,g^{\alpha\beta}(\partial_{x_{\alpha}} \omega_{\beta} + \omega_{\alpha} \omega_{\beta} -
\sum_{\gamma =1}^2\,\Gamma^{\gamma}_{\alpha\beta} \omega_{\gamma}) + E}{(\mu - 2 \sqrt{|\xi|^{2} + \lambda})^{2} \sqrt{|\xi|^{2} + \lambda}}\bigg|_{x=p}  \\
& = & - \frac{ \sum_{\alpha , \beta =1}^2\, (\partial_{x_{\alpha}} \omega_{\alpha} + \omega_{\alpha} \omega_{\alpha}) + E}{(\mu - 2 \sqrt{|\xi|^{2} + \lambda})^{2} \sqrt{|\xi|^{2} + \lambda}}\bigg|_{x=p},
\end{eqnarray*}
\noindent
which leads to
\begin{eqnarray*}
& &  \frac{1}{(2 \pi)^{2}}\int_{T_{p}^{\ast}{\mathcal N}}
 \frac{1}{2 \pi i} \int_{\gamma} \mu^{-s} ~ \Tr (7)\left(x, \xi, \frac{\lambda}{|\lambda|}, \mu\right) ~ d\mu ~ d \xi  \\
& = & \frac{1}{4 \pi} \Tr \left\{\sum_{\alpha =1}^2\,(\partial_{x_{\alpha}} \omega_{\alpha}\big|_{x=p} + \omega_{\alpha} \omega_{\alpha}\big|_{x=p}) + E(p) \right\} \cdot 2^{-s}.
\end{eqnarray*}
\noindent
Hence,
\begin{eqnarray}   \label{E:3.255}
& & \frac{1}{(2 \pi)^{2}}\int_{T_{p}^{\ast}{\mathcal N}}
\frac{1}{2 \pi i} \int_{\gamma} \mu^{-s} ~ \Tr (\Ee)\left(x, \xi, \frac{\lambda}{|\lambda|}, \mu \right) ~ d\mu ~ d \xi   \\
& = & r_{0} \cdot 2^{-s} \left\{
\frac{\tau_{{\mathcal N}}(p)}{24 \pi} \frac{1}{s+2} ~ + ~ \frac{1}{16 \pi} \left( \tau_{M}(p) - \tau_{{\mathcal N}}(p) \right)  \frac{s+1}{s+2}
~ + ~ \frac{H_{1}^{2}(p)}{4 \pi} \frac{(s-1)(s+1)}{(s+2)(s+4)}
~ - ~ \frac{H_{2}(p)}{8 \pi}  \frac{s-1}{s+4} \right\}  \nonumber  \\
& & + ~ \frac{1}{4 \pi} \Tr \sum_{\alpha =1}^2 \left( \partial_{x_{\alpha}} \omega_{\alpha} + \omega_{\alpha} \omega_{\alpha} \right)\big|_{x=p} \cdot 2^{-s}  \frac{s+1}{s+2}
  ~ + ~ \frac{1}{4 \pi} \left( \Tr E(p) \right) \cdot 2^{-s}.  \nonumber
\end{eqnarray}
\noindent
Summing up from (\ref{E:3.22}) to (\ref{E:3.255}), we have
\begin{eqnarray}    \label{E:3.27}
J_{2}(s, p) & := & \frac{1}{(2 \pi)^{2}} \int_{T_{p}^{\ast}{\mathcal N}}
\frac{1}{2 \pi i} \int_{\gamma} \mu^{-s} \Tr r_{-1-2} \left(p, \xi, \frac{\lambda}{|\lambda|}, \mu\right) ~ d\mu ~ d \xi \\
& = & r_{0} \cdot 2^{-s} \left\{
\frac{\tau_{{\mathcal N}}(p)}{24 \pi} ~ + ~ \frac{1}{16 \pi} \left( \tau_{M}(p) - \tau_{{\mathcal N}}(p) \right) \frac{s+1}{s+2}
 ~ + ~ \frac{H_{1}^{2}(p)}{4 \pi} \frac{(s-1)(s+1)}{(s+2)(s+4)}
- \frac{H_{2}(p)}{8 \pi} \frac{s-1}{s+4} \right\}  \nonumber  \\
& & ~ + ~ \frac{1}{4 \pi} \left( \Tr E(p) \right) \cdot 2^{-s}.  \nonumber
\end{eqnarray}
\noindent
This leads to the following result, which is one of the main results in this paper.
\begin{theorem} \label{Theorem:3.3}
Let ${\mathcal N}$ be a $2$-dimensional closed Riemannian manifold. Then the densities $q_{2}(p)$ and $\pi_{2}(p)$ are given as follows:
\begin{eqnarray*}
 q_{2}(p) & = & \frac{1}{2} J_{2}(s, p)\big|_{s=0} ~ = ~ \frac{1}{2} \cdot r_{0} \left\{ \frac{\tau_{M}(p)}{32 \pi} ~ + ~ \frac{\tau_{{\mathcal N}}(p)}{96 \pi}
~ - ~ \frac{H_{1}^{2}(p)}{32 \pi} ~ + ~ \frac{H_{2}(p)}{32 \pi} \right\} ~ + ~ \frac{1}{8 \pi} \left( \Tr E(p) \right) , \\
\pi_{2}(p) & = & - \frac{\partial}{\partial s}\big|_{s=0} J_{2}(s, p) \\
& = & r_{0} \left\{
\left(  \frac{\tau_{M}(p)}{32 \pi} ~ + ~ \frac{\tau_{{\mathcal N}}(p)}{96 \pi} ~ - ~ \frac{H_{1}^{2}(p)}{32 \pi} ~ + ~ \frac{H_{2}(p)}{32 \pi}  \right) \ln 2
~ - ~ \frac{1}{64 \pi} \left( \tau_{M}(p) - \tau_{{\mathcal N}}(p) \right) \right.    \\
& & \left. \hspace{2.0 cm} ~ - ~ \frac{3 H_{1}^{2}(p)}{128 \pi} ~ + ~ \frac{5 H_{2}(p)}{128 \pi} \right\}
 ~ + ~ \frac{1}{4 \pi} \left( \Tr E(p) \right) \cdot  \ln 2.   \nonumber
\end{eqnarray*}
\end{theorem}
\vspace{0.2 cm}
If $g$ is a product metric and $\Delta_{M} = - \partial_{x_{3}}^{2} + \Delta_{{\mathcal N}}$ on a collar neighborhood of ${\mathcal N}$ with $\Delta_{{\mathcal N}}$ a Laplacian on ${\mathcal N}$, then we have the
following result,
which is already obtained in \cite{Le2,PW}.
\begin{corollary}  \label{Corollary:3.5}
If $g$ is a product metric and $\Delta_{M} = - \partial_{x_{3}}^{2} + \Delta_{{\mathcal N}}$ on a collar neighborhood of ${\mathcal N}$,  we get
$\pi_{2} = \ln 2 \cdot \left( \zeta_{\Delta_{{\mathcal N}}}(0) + \Dim \Ker \Delta_{{\mathcal N}} \right)$.
\end{corollary}
\begin{proof}
If $g$ is a product metric near ${\mathcal N}$, it is well known that $H_{1}(p) = H_{2}(p) = 0$ and $\tau_{M}(p) = \tau_{{\mathcal N}}(p)$.
Hence, Theorem \ref{Theorem:3.3} shows that $\pi_{2}(p) = \left( r_{0} \cdot \frac{\tau_{{\mathcal N}}(p)}{24 \pi} + \frac{1}{4 \pi} \cdot \Tr E(p) \right) \ln 2$,
which together with Theorem 3.3.1 in \cite{Gi2} yields the result.
\end{proof}
Comparing the asymptotic expansions in (\ref{E:3.13}), we have the following equalities by (\ref{E:3.15}) and (\ref{E:3.19}).
\begin{eqnarray}     \label{E:3.28}
a_{0}(p) & = & - \pi_{2}(p),   \\
a_{1}(p) & = & c_{1}(p) - \pi_{0}(p) ~ = ~ r_{0} \left\{ \frac{1}{8 \pi} - \left( - \frac{1}{4 \pi} \ln 2 + \frac{1}{8 \pi} \right) \right\}
~ = ~ r_{0} \cdot \frac{1}{4 \pi} \ln 2.   \nonumber
\end{eqnarray}
\noindent
This leads to the following result.
\begin{corollary}  \label{Corollary:3.4}
Let ${\mathcal N}$ be a $2$-dimensional closed Riemannian manifold. Then $P(\lambda)$ in (\ref{E:1.4}) is given as follows.
\begin{eqnarray*}
P(\lambda) & = & a_{1} \lambda + a_{0} ~ = ~ \int_{{\mathcal N}} \left(  r_{0} \cdot \frac{\ln 2}{4 \pi} ~ \lambda ~ - ~ \pi_{2}(p) \right) ~ d \vol({\mathcal N}) \\
& = & \frac{r_{0} \cdot \vol({\mathcal N})  \ln 2}{4 \pi} ~ \lambda ~ - ~ \int_{{\mathcal N}} \pi_{2}(p)  ~ d \vol({\mathcal N}) .
\end{eqnarray*}
\end{corollary}
{\it Remark:} In \cite{KL3} we consider the gluing formula for the Laplacian on a warped product manifold
$M = [a,b]\times_f {\mathcal N}$, where $f:[a,b]\to {\mathbf R}$ is a positive smooth function. In that case
$r_0 =1$, $E(p)=0$, and the metric tensor is given by
\begin{eqnarray}    \label{E:5.1}
g(x,x_{3}) & = & \left( \begin{array}{clcr} f(x_{3})^{2} h_{\alpha\beta}(x) & 0 \\ 0 & 1\end{array} \right),
\end{eqnarray}
where $h=(h_{\alpha \beta} (x))$ is a metric tensor on ${\mathcal N}$. In that context let $a<c<b$. We assume
$f(c)=1$ and identify $\{c \} \times {\mathcal N}$ with ${\mathcal N}$. Corollary 3.7 of \cite{KL3} then implies
\begin{eqnarray}
P(\lambda ) &=& \frac{ \mbox{vol} ( {\mathcal N}) \cdot \ln 2} {4\pi} \lambda - \frac {\ln 2} {24 \pi} \int_{\mathcal N} \tau_{\mathcal N} (p) d \mbox{vol}
({\mathcal N})\label{E:5.2}\\
&&+\frac{ \mbox{vol} ({\mathcal N})} {4\pi} \left\{ \ln 2 \left( \frac 1 2 f'' (c) + \frac 1 4 f' (c)^2 \right)
- \left( \frac 3 {16} f' (c) ^2 + \frac 1 4 f'' (c) \right) \right\} .\nonumber
\end{eqnarray}
Using the metric tensor (\ref{E:5.1}) one can show that
$$\kappa_1 = \kappa_2 = - f' (c),$$
furthermore
$$ \tau _M (p) - \tau _{\mathcal N} (p) =  - 4 f'' (c) - 2 f' (c)^2.$$
With the information given, Corollary \ref{Corollary:3.4} is then seen to agree with (\ref{E:5.2}), confirming our previous results.

Before finishing this section, we make one observation.
If we compare the coefficients of $\ln \lambda$ in the asymptotic expansions in (\ref{E:3.13}), we see that $q_{2} = c_{3}$.
This fact is shown explicitly in Lemma \ref{Lemma:3.2} and Theorem \ref{Theorem:3.3}.
\section{The heat trace expansion of $e^{- t R(\lambda)}$ for $t \rightarrow 0^{+}$}
We begin this section with the following well known asymptotic expansion for $t \rightarrow 0^{+}$,
\begin{eqnarray}    \label{E:4.1}
\Tr e^{-t R(\lambda)} & \sim & \sum_{j=0}^{\infty} v_{j}(\lambda) t^{-m + j} + \sum_{l=1}^{\infty} w_{l}(\lambda) t^{l} \ln t,
\end{eqnarray}
\noindent
where $v_{j}(\lambda)$ for $0 \leq j \leq m$ is computed by integration of some local density $v_{j}(\lambda)(p)$ involving the homogeneous symbol of $(\mu - R(\lambda))^{-1}$
\cite{Li,PS}. As in Section 3, ultimately we will assume that $m = 3$ and that we are going to compute the densities $v_{0}(\lambda)(p)$, $v_{1}(\lambda)(p)$
and $v_{2}(\lambda)(p)$ for $p \in {\mathcal N}$
 in terms of the scalar and principal curvatures of ${\mathcal N}$ in $M$. However, at the beginning we will keep $m$ general.
Here we should mention that in this section we regard $\lambda$ as a usual constant, not as a parameter of weight $2$ as in the previous section.
Hence we should put the symbol of the tangential Laplacian as (cf. \ref{E:2.17})
\begin{eqnarray}   \label{E:4.2}
\sigma \left(D \left( x, x_{m}, \frac{\partial}{\partial x}\right) + \lambda \right) & = & {\widetilde p}_{2}(x, x_{m}, \xi) ~ + ~
{\widetilde p}_{1}(x, x_{m}, \xi) ~ + ~ {\widetilde p}_{0}(x, x_{m}, \xi, \lambda),
\end{eqnarray}
\noindent
where
\begin{eqnarray}   \label{E:4.3}
& & {\widetilde p}_{2}(x, x_{m}, \xi) ~ = ~ \sum_{\alpha, \beta = 1}^{m-1} g^{\alpha\beta}(x,  x_{m}) \xi_{\alpha} \xi_{\beta} \Id
 ~ = ~ |\xi|^{2} \Id,   \\
& & {\widetilde p}_{1}(x,  x_{m}, \xi) ~ = ~ - i
\sum_{\alpha, \beta = 1}^{m-1} \left( \frac{1}{2} g^{\alpha\beta}(x,  x_{m}) \partial_{ x_{\alpha}} \ln |g|(x,  x_{m}) +
\partial_{ x_{\alpha}} g^{\alpha\beta}(x,  x_{m}) \right) \xi_{\beta} \Id
- 2 i \sum_{\alpha, \beta = 1}^{m-1} g^{\alpha\beta} \omega_{\alpha} \xi_{\beta} ,   \nonumber \\
& & {\widetilde p}_{0}(x,  x_{m}, \xi, \lambda) ~ = ~ \lambda ~ \Id
~ - ~ \sum_{\alpha , \beta =1}^{m-1} \, g^{\alpha\beta} \left( \partial_{x_{\alpha}} \omega_{\beta} + \omega_{\alpha} \omega_{\beta} -
\sum_{\gamma =1} ^{m-1} \, \Gamma_{\alpha\beta}^{\gamma} \omega_{\gamma} \right) - E . \nonumber
\end{eqnarray}
\noindent
As in (\ref{E:2.16}) and (\ref{E:2.22}), we denote the homogeneous symbols of $Q_{x_{m}}^{\pm}(\lambda)$ and $R(\lambda)$ as
\begin{eqnarray}   \label{E:4.44}
& & \sigma \left( Q^{+}_{x_{m}}(\lambda) \right) ~ \sim ~ {\widetilde \alpha}_{1}(x, x_{m}, \xi, \lambda) ~ + ~
{\widetilde \alpha}_{0}(x, x_{m}, \xi, \lambda) ~ + ~ {\widetilde \alpha}_{-1}(x, x_{m}, \xi, \lambda) ~ + ~  \cdots,   \nonumber  \\
& & \sigma \left( Q^{-}_{x_{m}}(\lambda) \right) ~ \sim ~ {\widetilde \beta}_{1}(x, x_{m}, \xi, \lambda) ~ + ~
{\widetilde \beta}_{0}(x, x_{m}, \xi, \lambda) ~ + ~ {\widetilde \beta}_{-1}(x, x_{m}, \xi, \lambda) ~ + ~  \cdots,   \nonumber    \\
& & \sigma(R(\lambda))(x, 0, \xi, \lambda)  ~ \sim ~ {\widetilde \theta}_{1}(x, 0, \xi, \lambda)
+ {\widetilde \theta}_{0}(x, 0, \xi, \lambda) + {\widetilde \theta}_{-1}(x, 0, \xi, \lambda) + \cdots,
\end{eqnarray}
\noindent
where ${\widetilde \theta}_{j}(x, 0, \xi, \lambda) = {\widetilde \alpha}_{j}(x, 0, \xi, \lambda) +
{\widetilde \beta}_{j}(x, 0, \xi, \lambda)$.
As in Section 3, we use Lemma \ref{Lemma:2.3}, (\ref{E:4.2}) and (\ref{E:4.3}) to compute ${\widetilde \alpha}_{j}(x, x_{m}, \xi, \lambda)$ and ${\widetilde \beta}_{j}(x, x_{m}, \xi, \lambda)$, recursively.
By adding them, we compute ${\widetilde \theta}_{j}(x, 0, \xi, \lambda)$. The only difference is that $\lambda$ belongs to the
zero order term in the homogeneous symbol of the tangential Laplacian $D ( x, x_{m}, \frac{\partial}{\partial x}) + \lambda $.
For example, the first three terms are given as follows (cf.~(\ref{E:2.20}), (\ref{E:2.21}))
(we omit $\Id$ to simplify the presentation as before):
\begin{eqnarray*}
& & {\widetilde \alpha}_{1}(x, x_{m}, \xi, \lambda) = {\widetilde \beta}_{1}(x, x_{m}, \xi, \lambda) = |\xi| = \alpha_{1}(x, x_{m}, \xi, 0) = \beta_{1}(x, x_{m}, \xi, 0),  \\
& & {\widetilde \alpha}_{0}(x, x_{m}, \xi, \lambda) = \alpha_{0}(x, x_{m}, \xi, 0), \qquad
{\widetilde \beta}_{0}(x, x_{m}, \xi, \lambda) = \beta_{0}(x, x_{m}, \xi, 0),  \\
& & {\widetilde \alpha}_{-1}(x, x_{m}, \xi, \lambda) ~ = ~ \frac{1}{2 |\xi|}
\Big\{\sum_{|\omega|=2} \frac{1}{\omega !} (\partial_{\xi}^{\omega} {\widetilde \alpha}_{1} )(\partial_{x}^{\omega} {\widetilde \alpha}_{1} )+
i (\partial_{\xi} {\widetilde \alpha}_{0})( \partial_{x} {\widetilde \alpha}_{1} )+ i (\partial_{\xi} {\widetilde \alpha}_{1})( \partial_{x} {\widetilde \alpha}_{0})
\nonumber\\
& &\hspace{1.0cm} - {\widetilde \alpha}_{0}^{2} + {\widetilde p}_{0} + A(x, x_{m}) {\widetilde \alpha}_{0}
  - \partial_{x_{m}} {\widetilde \alpha}_{0} + ({\widetilde \alpha}_{0} \omega_{m} - \omega_{m} {\widetilde \alpha}_{0}) \Big\},   \\
& & {\widetilde \beta}_{-1}(x, x_{m}, \xi, \lambda) ~ = ~ \frac{1}{2 |\xi|}
\Big\{ \sum_{ | \omega | =2} \frac{1}{\omega !} (\partial_{\xi}^{\omega} {\widetilde \beta}_{1} )(\partial_{x}^{\omega} {\widetilde \beta}_{1} )+
i (\partial_{\xi} {\widetilde \beta}_{0})( \partial_{x} {\widetilde \beta}_{1} )+ i (\partial_{\xi} {\widetilde \beta}_{1})( \partial_{x} {\widetilde \beta}_{0}) \nonumber\\
& & \hspace{1.0cm} - {\widetilde \beta}_{0}^{2} + {\widetilde p}_{0} - A(x, x_{m}) {\widetilde \beta}_{0}
 + \partial_{x_{m}} {\widetilde \beta}_{0} - ({\widetilde \beta}_{0} \omega_{m} - \omega_{m} {\widetilde \beta}_{0}) \Big\}.
\end{eqnarray*}
\noindent
We put
\begin{eqnarray}   \label{E:4.5}
\sigma \left( (\mu - R(\lambda))^{-1} \right) & \sim & {\widetilde r}_{-1}(x, \xi, \lambda, \mu) +
{\widetilde r}_{-2}(x, \xi, \lambda, \mu) +
{\widetilde r}_{-3}(x, \xi, \lambda, \mu) + \cdots,
\end{eqnarray}
\noindent
where
\begin{eqnarray}   \label{E:4.6}
{\widetilde r}_{-1}(x, \xi, \lambda, \mu) & = & \left( \mu - 2 |\xi| \right)^{-1},  \\
{\widetilde r}_{-1-j}(x, \xi, \lambda, \mu) & = & \left( \mu - 2 |\xi| \right)^{-1} \sum_{k=0}^{j-1} \sum_{|\omega|+l+k=j} \frac{1}{\omega!}
\partial_{\xi}^{\omega} {\widetilde \theta}_{1-l} \cdot  D_{x}^{\omega} {\widetilde r}_{-1-k},    \nonumber
\end{eqnarray}
\noindent
which shows that the first three terms are given as follows (cf. (\ref{E:2.27})),
\begin{eqnarray}   \label{E:4.7}
{\widetilde r}_{-1} & = & \left( \mu - 2 |\xi| \right)^{-1},  \\
{\widetilde r}_{-2} & = & \left( \mu - 2 |\xi| \right)^{-1}
\left\{ \partial_{\xi} {\widetilde \theta}_{1} \cdot D_{x} {\widetilde r}_{-1} + {\widetilde \theta}_{0} \cdot {\widetilde r}_{-1} \right\}, \nonumber  \\
{\widetilde r}_{-3} & = & \left( \mu - 2 |\xi| \right)^{-1} \left\{ \sum_{ | \omega | =2} \frac{1}{\omega !} \partial_{\xi}^{\omega} {\widetilde \theta}_{1} \cdot D_{x}^{\omega} {\widetilde r}_{-1}
+ \partial_{\xi} {\widetilde \theta}_{1} \cdot D_{x} {\widetilde r}_{-2} +
\partial_{\xi} {\widetilde \theta}_{0} \cdot D_{x} {\widetilde r}_{-1} + {\widetilde \theta}_{0} \cdot {\widetilde r}_{-2}
 + {\widetilde \theta}_{-1} \cdot {\widetilde r}_{-1} \right\} \nonumber   \\
& =: & ({\widetilde \Aa}) + ({\widetilde \Bb}) +({\widetilde \Cc}) +({\widetilde \Dd}) +({\widetilde \Ee}). \nonumber
\end{eqnarray}
\noindent
Then, the density $v_{j}(\lambda)(x)$ ($j = 0, 1, 2$) computing $v_{j}(\lambda)$ in (\ref{E:4.1}) is given as \cite{Gi1,PS}
\begin{eqnarray}     \label{E:4.8}
v_{j}(\lambda)(x) ~ = ~ \frac{1}{(2 \pi)^{2}} \int_{T^{\ast}_{x}{\mathcal N}}
\frac{1}{2 \pi i} \int_{\gamma} e^{- \mu} \Tr {\widetilde r}_{-1-j} (x, \xi, \lambda, \mu) d\mu d\xi .
\end{eqnarray}
\noindent
We now use the above formula to compute $v_{j}(\lambda)(x)$ ($j = 0, 1, 2$) at $x = p \in {\mathcal N}$.
The following lemma is straightforward.
\begin{lemma}   \label{Lemma:4.33}
Let ${\mathbb C}^+ = {\mathbb C} -\{r\in {\mathbb R} | r\leq 0\}$. For $z\in {\mathbb C}^+$ let $\gamma$ be a counterclockwise contour in ${\mathbb C}^+$
with $z$ inside $\gamma$.
Then
\begin{eqnarray*}
& (1) & \frac{1}{2 \pi i} \int_{\gamma} \frac{e^{-\mu}}{\mu - z} d \mu = e^{- z}, \qquad  \frac{1}{2 \pi i} \int_{\gamma} \frac{e^{-\mu}}{(\mu - z)^{2}} d \mu = - e^{- z}, \qquad
\frac{1}{2 \pi i} \int_{\gamma} \frac{e^{-\mu}}{(\mu - z)^{3}} d \mu = \frac{1}{2} e^{- z},   \\
& (2) & \frac{1}{4 \pi^{2}} \int_{{\mathbb R}^{2}} e^{-2|\xi|} d \xi ~ = ~ \frac{1}{8 \pi}, \qquad
\frac{1}{4 \pi^{2}} \int_{{\mathbb R}^{2}} \frac{e^{-2|\xi|}}{|\xi|} d \xi ~ = ~ \frac{1}{4 \pi}, \qquad
 \frac{1}{4 \pi^{2}} \int_{{\mathbb R}^{2}} \frac{e^{-2|\xi|} \xi_{1}^{2}}{|\xi|^{2}} d \xi ~ = ~ \frac{1}{16 \pi}, \\
& & \frac{1}{4 \pi^{2}} \int_{{\mathbb R}^{2}} \frac{e^{-2|\xi|} \xi_{1}^{2}}{|\xi|^{3}} d \xi ~ = ~ \frac{1}{8 \pi}, \qquad
\frac{1}{4 \pi^{2}} \int_{{\mathbb R}^{2}} \frac{e^{-2|\xi|} \xi_{1}^{4}}{|\xi|^{5}} d \xi ~ = ~ \frac{3}{32 \pi}, \qquad\frac{1}{4 \pi^{2}} \int_{{\mathbb R}^{2}} \frac{e^{-2|\xi|} \xi_{1}^{2} \xi_{2}^{2}}{|\xi|^{5}} d \xi ~ = ~ \frac{1}{32 \pi}.
\end{eqnarray*}
\end{lemma}
\subsection{The coefficient $v_{0}(\lambda)$}
\begin{eqnarray}   \label{E:4.9}
v_{0}(\lambda)(p) ~ = ~ \frac{r_{0}}{(2 \pi)^{2}} \int_{T^{\ast}_{p}{\mathcal N}}
\frac{1}{2 \pi i} \int_{\gamma} \frac{e^{- \mu}}{\mu - 2 |\xi|} d\mu d\xi
~ = ~ \frac{r_{0}}{(2 \pi)^{2}} \int_{T^{\ast}_{p}{\mathcal N}} e^{- 2 |\xi|} d\xi ~ = ~ \frac{r_{0}}{8 \pi},
\end{eqnarray}
\noindent
which shows that
\begin{eqnarray}    \label{E:4.10}
v_{0}(\lambda) & = & \frac{r_{0}}{8 \pi} \vol({\mathcal N}).
\end{eqnarray}
\subsection{The coefficient $v_{1}(\lambda)$}
Since ${\widetilde r}_{-2} (x, \xi, \lambda, \mu)$ is an odd function with respect to $\xi$,
$v_{1}(\lambda)(p) = 0$ and hence $v_{1}(\lambda) = 0$.
\subsection{The coefficient $v_{2}(\lambda)$}
We note that
\begin{eqnarray}   \label{E:4.12}
v_{2}(\lambda)(p) ~ = ~ \frac{1}{(2 \pi)^{2}} \int_{T^{\ast}_{p}{\mathcal N}}
\frac{1}{2 \pi i} \int_{\gamma} e^{- \mu}  \Tr {\widetilde r}_{-3} (x, \xi, \lambda, \mu)\big|_{x=p} ~  d\mu ~ d\xi,
\end{eqnarray}
\noindent
where ${\widetilde r}_{-3} (x, \xi, \lambda, \mu)|_{x=p}  = ({\widetilde \Aa}) + ({\widetilde \Bb}) +({\widetilde \Cc}) +({\widetilde \Dd}) +({\widetilde \Ee})$.
We compute $({\widetilde \Aa}) \sim ({\widetilde \Ee})$  as $(\Aa) \sim (\Ee)$ in Section 3 to obtain the following.
\begin{eqnarray*}
({\widetilde \Aa}) & = & (\mu - 2 |\xi|)^{-1} \left( \sum_{|\omega |=2} \frac{1}{\omega !} \partial_{\xi}^{\omega} {\widetilde \theta}_{1} \cdot D_{x}^{\omega} {\widetilde r}_{-1} \right)\big|_{x=p}  \\
& = & (\mu - 2 |\xi|)^{-1} \left\{ \frac{1}{2}
\sum_{\alpha =1}^2\, \partial_{\xi_{\alpha}}^{2} \left( 2 |\xi|  \right) \cdot (-1)
\partial_{x_{\alpha}}^{2} \left(\mu - 2 |\xi| \right)^{-1}
 + \partial_{\xi_{1}} \partial_{\xi_{2}} \left( 2 |\xi| \right) \cdot (-1)
\partial_{x_{1}} \partial_{x_{2}} \left(\mu - 2 |\xi|  \right)^{-1} \right\}\big|_{x=p}   \\
& \approx & \frac{- \frac{2}{3} ~ \tau_{{\mathcal N}}(p) ~ \xi_{1}^{2}}{(\mu - 2 |\xi|)^{3} |\xi|^{2} } ,
\end{eqnarray*}
\noindent
which leads to
\begin{eqnarray}    \label{E:4.13}
\frac{1}{(2 \pi)^{2}}\int_{T_{p}^{\ast}{\mathcal N}}
\frac{1}{2 \pi i} \int_{\gamma} e^{-\mu} ~ \Tr ({\widetilde \Aa}) ~ d\mu ~ d \xi & = & - r_{0} \cdot \frac{\tau_{{\mathcal N}}(p)}{48 \pi}.
\end{eqnarray}
\begin{eqnarray*}
({\widetilde \Bb}) & = & \frac{\partial_{\xi} {\widetilde \theta}_{1} \cdot D_{x} {\widetilde r}_{-2}}{\mu - 2 |\xi|}\big|_{x=p}
~ = ~ \frac{(-i) \sum_{\alpha =1}^2\, \partial_{\xi_{\alpha}} ( 2 |\xi| ) \cdot \partial_{x_{\alpha}}
\left( ( \mu - 2 |\xi|)^{-1} ( \sum_{\beta =1}^2\,
\partial_{\xi_{\beta}} {\widetilde \theta}_{1} \cdot D_{x_{\beta}} {\widetilde r}_{-1} + {\widetilde \theta}_{0} \cdot {\widetilde r}_{-1}) \right)}{\mu - 2 |\xi|}\big|_{x=p}   \\
& \approx & \frac{\frac{4}{3} \tau_{{\mathcal N}}(p) \xi_{1}^{2} - 4 \sum_{\alpha =1}^2(\partial_{x_{\alpha}} \omega_{\alpha}) \xi_{1}^{2}}{(\mu - 2 |\xi|)^{3} |\xi|^{2} } ,
\end{eqnarray*}
\noindent
which leads to
\begin{eqnarray}   \label{E:4.14}
\frac{1}{(2 \pi)^{2}}\int_{T_{p}^{\ast}{\mathcal N}}
\frac{1}{2 \pi i} \int_{\gamma} e^{-\mu} ~ \Tr ({\widetilde \Bb}) ~ d\mu ~ d \xi & = &
r_{0} \cdot \frac{\tau_{{\mathcal N}}(p)}{24 \pi}
- \frac{1}{8 \pi} \Tr \left( \sum_{\alpha =1}^2\, \partial_{x_{\alpha}} \omega_{\alpha}\big|_{x=p}  \right).
\end{eqnarray}
\begin{eqnarray}    \label{E:4.15}
({\widetilde \Cc}) & = & \frac{\sum_{\alpha =1}^2\, \partial_{\xi_{\alpha}} {\widetilde \theta}_{0} \cdot D_{x_{\alpha}} {\widetilde r}_{-1}}{\mu - 2 |\xi|}\big|_{x=p}
~ = ~ \frac{(-i) \sum_{\alpha =1}^2\, \partial_{\xi_{\alpha}} {\widetilde \theta}_{0} \cdot \partial_{x_{\alpha}}
(\mu - 2 |\xi|)^{-1}}{\mu - 2 |\xi|}\big|_{x=p}  \nonumber \\
& = & \frac{(-i) \sum_{\alpha =1}^2 \partial_{\xi_{\alpha}} {\widetilde \theta}_{0} \cdot \sum_{\gamma ,\delta =1}^2\, g^{\gamma \delta, \alpha} \xi_{\gamma} \xi_{\delta}}{(\mu - 2 |\xi|)^{3} |\xi|}\big|_{x=p}
~ = ~ 0,
\end{eqnarray}
\noindent
since $g^{\gamma \delta, \alpha}(p) = 0$.
\begin{eqnarray*}
({\widetilde \Dd}) & = & \frac{{\widetilde \theta}_{0} \cdot {\widetilde r}_{-2}}{\mu - 2 |\xi|}\big|_{x=p}  ~ = ~
\frac{{\widetilde \theta}_{0}^{2}}{(\mu - 2 |\xi|)^{3}}\big|_{x=p}  ~ \approx ~
\frac{(-4 \sum_{\alpha =1}^2\, \omega_{\alpha} \omega_{\alpha}) \xi_{1}^{2}}{(\mu - 2 |\xi|)^{3} |\xi|^{2}}\big|_{x=p}.
\end{eqnarray*}
\begin{eqnarray}    \label{E:4.17}
\frac{1}{(2 \pi)^{2}}\int_{T_{p}^{\ast}{\mathcal N}}
\frac{1}{2 \pi i} \int_{\gamma} e^{-\mu} ~ \Tr ({\widetilde \Dd}) ~ d\mu ~ d \xi & = &
- \frac{1}{8 \pi} \Tr \left( \sum_{\alpha =1}^2\, \omega_{\alpha} \omega_{\alpha}\big|_{x=p}  \right).
\end{eqnarray}
\begin{eqnarray*}
({\widetilde \Ee}) & = & \frac{{\widetilde \theta}_{-1} \cdot {\widetilde r}_{-1}}{\mu - 2 |\xi|}\big|_{x=p}  ~ = ~
\frac{{\widetilde \theta}_{-1}}{(\mu - 2 |\xi|)^{2}}\big|_{x=p}  ~ = ~
\frac{1}{2(\mu - 2 |\xi|)^{2} |\xi|} \left\{- \sum_{| \omega | =2} \frac{2}{\omega !} \partial_{\xi}^{\omega} {\widetilde \alpha}_{1} \cdot D_{x}^{\omega} {\widetilde \alpha}_{1} \right.  \\
&  &\left. - \partial_{\xi} {\widetilde \theta}_{0} \cdot D_{x} {\widetilde \alpha}_{1} - \partial_{\xi} {\widetilde \alpha}_{1} \cdot D_{x} {\widetilde \theta}_{0}
- ({\widetilde \alpha}_{0}^{2} + {\widetilde \beta}_{0}^{2}) + A(x,0) ({\widetilde \alpha}_{0} - {\widetilde \beta}_{0}) - \partial_{x_{m}} ({\widetilde \alpha}_{0} - {\widetilde \beta}_{0}) + 2 {\widetilde p}_{0} \right\}\big|_{x=p}   \\
& =: & ({\widetilde 1}) + ({\widetilde 2}) + ({\widetilde 3}) + ({\widetilde 4}) + ({\widetilde 5}) + ({\widetilde 6}) + ({\widetilde 7}).
\end{eqnarray*}
\begin{eqnarray*}
({\widetilde 1}) & = & \frac{- \sum_{ | \omega | =2} \frac{2}{\omega !} \partial_{\xi}^{\omega} {\widetilde \alpha}_{1}
\cdot D_{x}^{\omega} {\widetilde \alpha}_{1}}{2(\mu - 2 |\xi|)^{2} |\xi|}\big|_{x=p}
~ \approx \frac{\tau_{{\mathcal N}}(p)}{6} \cdot \frac{\xi_{1}^{2}}{(\mu - 2 |\xi|)^{2} |\xi|^{3}} ,
\end{eqnarray*}
\noindent
which leads to
\begin{eqnarray*}
\frac{1}{(2 \pi)^{2}}\int_{T_{p}^{\ast}{\mathcal N}}
\frac{1}{2 \pi i} \int_{\gamma} e^{-\mu} ~ \Tr ({\widetilde 1}) ~ d\mu ~ d \xi & = & - r_{0} \cdot \frac{\tau_{{\mathcal N}}(p)}{48 \pi}.
\end{eqnarray*}
\begin{eqnarray*}
({\widetilde 2}) & = & \frac{- \partial_{\xi} {\widetilde \theta}_{0} \cdot D_{x} {\widetilde \alpha}_{1}}
{2(\mu - 2 |\xi|)^{2} |\xi|}\big|_{x=p}
~ = ~ \frac{i \sum_{\alpha =1}^2\, \partial_{\xi_{\alpha}} {\widetilde \theta}_{0} \cdot \sum_{\beta ,\gamma =1}^2\, g^{\beta\gamma,\alpha} \xi_{\beta} \xi_{\gamma}}{4(\mu - 2 |\xi|)^{2} |\xi|^{2}}\big|_{x=p}
~ = ~ 0,
\end{eqnarray*}
\noindent
since $g^{\beta\gamma,\alpha}(p) = 0$.
\begin{eqnarray*}
({\widetilde 3}) & = & \frac{- \sum_{\alpha =1}^2\, \partial_{\xi_{\alpha}} {\widetilde \alpha}_{1} \cdot D_{x_{\alpha}} {\widetilde \theta}_{0}}{2(\mu - 2 |\xi|)^{2} |\xi|}\big|_{x=p}
~ \approx ~ \frac{- \frac{1}{3} \tau_{{\mathcal N}}(p) \xi_{1}^{2}}{(\mu - 2 |\xi|)^{2} |\xi|^{3}} \Id
~ + ~ \frac{(\sum_{\alpha =1}^2\, \partial_{x_{\alpha}} \omega_{\alpha}) \xi_{1}^{2} }{(\mu - 2 |\xi|)^{2} |\xi|^{3}}\big|_{x=p}
\end{eqnarray*}
\noindent
which leads to
\begin{eqnarray*}
\frac{1}{(2 \pi)^{2}}\int_{T_{p}^{\ast}{\mathcal N}}
\frac{1}{2 \pi i} \int_{\gamma} e^{-\mu} ~ \Tr ({\widetilde 3}) ~ d\mu ~ d \xi & = &
r_{0} \cdot \frac{\tau_{{\mathcal N}}(p)}{24 \pi}
- \frac{1}{8 \pi} \Tr \left( \sum_{\alpha =1}^2\, \partial_{x_{\alpha}} \omega_{\alpha}\big|_{x=p} \right).
\end{eqnarray*}
\begin{eqnarray*}
({\widetilde 4}) & = & \frac{- ({\widetilde \alpha}_{0}^{2} + {\widetilde \beta}_{0}^{2})}{2(\mu - 2 |\xi|)^{2} |\xi|}\big|_{x=p}
~ = ~ \left\{ \frac{\left( \sum_{\alpha =1}^2\, \omega_{\alpha} \xi_{\alpha} \right)^{2}}{(\mu - 2 |\xi|)^{2} |\xi|^{3}}
~ + ~ \frac{- \left( \frac{1}{2} A(x,0)^{2} - \frac{A(x,0)}{{\widetilde \alpha}_{1}} \partial_{x_{m}} {\widetilde \alpha}_{1} +
\frac{1}{2 {\widetilde \alpha}_{1}^{2}} \left( \partial_{x_{m}} {\widetilde \alpha}_{1} \right)^{2}\right)}{2(\mu - 2 |\xi|)^{2} |\xi|} \right\}\big|_{x=p}  \\
& \approx & \left\{ \frac{- H_{1}^{2}}{(\mu - 2 |\xi|)^{2} |\xi|} +
 \frac{(2H_{1}^{2}) \xi_{1}^{2} + (\sum_{\alpha =1}^2\, \omega_{\alpha} \omega_{\alpha}) \xi_{1}^{2}}{(\mu - 2 |\xi|)^{2} |\xi|^{3}}
~  - ~ \frac{(4 H_{1}^{2} - 2 H_{2}) \xi_{1}^{4} + 2 H_{2} \xi_{1}^{2} \xi_{2}^{2}}{4 (\mu - 2 |\xi|)^{2} |\xi|^{5}} \right\}\big|_{x=p} ,
\end{eqnarray*}
\noindent
which leads to
\begin{eqnarray*}
\frac{1}{(2 \pi)^{2}}\int_{T_{p}^{\ast}{\mathcal N}}
\frac{1}{2 \pi i} \int_{\gamma} e^{-\mu} ~ \Tr ({\widetilde 4}) ~ d\mu ~ d \xi
& = & r_{0} \left\{ \frac{3 H_{1}^{2}(p)}{32 \pi} - \frac{H_{2}(p)}{32 \pi} \right\}
 - \frac{1}{8 \pi} \Tr \left( \sum_{\alpha =1}^2\, \omega_{\alpha} \omega_{\alpha}\big|_{x=p} \right).
\end{eqnarray*}
\begin{eqnarray*}
({\widetilde 5}) & = & \frac{A(x,0) ({\widetilde \alpha}_{0} - {\widetilde \beta}_{0})}{2(\mu - 2 |\xi|)^{2} |\xi|}\big|_{x=p} ~ = ~
\frac{A(x,0)^{2} - \frac{A(x,0)}{{\widetilde \alpha}_{1}} \frac{\partial}{\partial _{x_m}} {\widetilde \alpha}_{1}}{2(\mu - 2 |\xi|)^{2} |\xi|}\big|_{x=p} \\
&  \approx & \left\{ \frac{ 2 H_{1}^{2}}{(\mu - 2 |\xi|)^{2} |\xi|} +
\frac{ -  2 H_{1}^{2} \xi_{1}^{2}}{(\mu - 2 |\xi|)^{2} |\xi|^{3}} \right\}\big|_{x=p} ,
\end{eqnarray*}
\noindent
which leads to
\begin{eqnarray*}
\frac{1}{(2 \pi)^{2}}\int_{T_{p}^{\ast}{\mathcal N}}
\frac{1}{2 \pi i} \int_{\gamma} e^{-\mu} ~ \Tr ({\widetilde 5}) ~ d\mu ~ d \xi
& = & - ~ r_{0} ~ \frac{H_{1}^{2}(p)}{4 \pi}.
\end{eqnarray*}
\begin{eqnarray*}
({\widetilde 6}) & = & \frac{- \partial_{x_{m}} ({\widetilde \alpha}_{0} - {\widetilde \beta}_{0})}{2(\mu - 2 |\xi|)^{2} |\xi|}\big|_{x=p} ~ = ~
\frac{- \partial_{x_{m}} \left( A(x,0) - \frac{1}{{\widetilde \alpha}_{1}} \partial_{x_{m}} {\widetilde \alpha}_{1} \right)}
{2(\mu - 2 |\xi|)^{2} |\xi|}\big|_{x=p}  \\
& \approx & \left\{ \frac{-4H_{1}^{2} + H_{2} - \frac{1}{2} (\tau_{M}(x) - \tau_{{\mathcal N}}(x))}{2(\mu - 2 |\xi|)^{2} |\xi|}
~ + ~ \frac{\left( 12 H_{1}^{2} - 5 H_{2} + \frac{1}{2} (\tau_{M}(x) - \tau_{{\mathcal N}}(x)) \right) \xi_{1}^{2}}{2(\mu - 2 |\xi|)^{2} |\xi|^{3}} \right. \\
& & \left. + ~ \frac{(-4H_{1}^{2} + 2H_{2}) \xi_{1}^{4} - 2 H_{2} \xi_{1}^{2} \xi_{2}^{2}}{(\mu - 2 |\xi|)^{2} |\xi|^{5}} \right\}\big|_{x=p} \Id,
\end{eqnarray*}
\noindent
which leads to
\begin{eqnarray*}
\frac{1}{(2 \pi)^{2}}\int_{T_{p}^{\ast}{\mathcal N}}
\frac{1}{2 \pi i} \int_{\gamma} e^{-\mu} ~ \Tr ({\widetilde 6}) ~ d\mu ~ d \xi
& = & r_{0} \left\{ \frac{H_{1}^{2}(p)}{8 \pi} + \frac{H_{2}(p)}{16 \pi}
~ + ~ \frac{1}{32 \pi}  \left( \tau_{M}(p) - \tau_{{\mathcal N}}(p) \right) \right\}.
\end{eqnarray*}
\begin{eqnarray*}
({\widetilde 7}) & = & \frac{2 {\widetilde p}_{0}}{2(\mu - 2 |\xi|)^{2} |\xi|}\big|_{x=p} ~ = ~  \frac{\lambda}{(\mu - 2 |\xi|)^{2} |\xi|}\big|_{x=p} ~ \Id
~ - ~ \frac{\sum_{\alpha , \beta =1}^2\, g^{\alpha\beta}(\partial_{x_{\alpha}} \omega_{\beta} + \omega_{\alpha} \omega_{\beta} - \sum_{\gamma =1}^2\,
\Gamma^{\gamma}_{\alpha\beta} \omega_{\gamma}) + E}{(\mu - 2 |\xi|)^{2} |\xi|}\big|_{x=p} \\
& = & \frac{\lambda}{(\mu - 2 |\xi|)^{2} |\xi|}\big|_{x=p} ~ \Id ~
- ~ \frac{\sum_{\alpha =1}^2 (\partial_{x_{\alpha}} \omega_{\alpha} + \omega_{\alpha} \omega_{\alpha} ) + E}{(\mu - 2 |\xi|)^{2} |\xi|}\big|_{x=p},
\end{eqnarray*}
\noindent
which leads to
\begin{eqnarray*}
\frac{1}{(2 \pi)^{2}}\int_{T_{p}^{\ast}{\mathcal N}} \frac{1}{2 \pi i} \int_{\gamma} e^{-\mu}  ~ \Tr ({\widetilde 7}) ~ d\mu ~ d \xi
& = & - \frac{r_{0}}{4 \pi} \lambda ~ + ~
\frac{1}{4 \pi} \Tr \left\{\sum_{\alpha =1}^2\, (\partial_{x_{\alpha}} \omega_{\alpha}\big|_{x=p} + \omega_{\alpha} \omega_{\alpha}\big|_{x=p}) + E(p) \right\}.
\end{eqnarray*}
\noindent
The sum of the above seven equalities leads to the following result.
\begin{eqnarray}   \label{E:4.16}
& & \frac{1}{(2 \pi)^{2}}\int_{T_{p}^{\ast}{\mathcal N}}
\frac{1}{2 \pi i} \int_{\gamma} e^{-\mu} ~ ({\widetilde \Ee}) ~ d\mu ~ d \xi
~ = ~ r_{0} \left\{ - \frac{1}{4 \pi} \lambda + \frac{\tau_{M}(p)}{32 \pi} - \frac{\tau_{{\mathcal N}}(p)}{96 \pi} ~ - ~ \frac{H_{1}^{2}(p)}{32 \pi}   ~ + ~ \frac{H_{2}(p)}{32 \pi} \right\}   \nonumber \\
& & + \frac{1}{8 \pi} \Tr \sum_{\alpha =1}^2\, \left( \partial_{x_{\alpha}} \omega_{\alpha}\big|_{x=p} + \omega_{\alpha} \omega_{\alpha}\big|_{x=p} \right) + \frac{1}{4 \pi} \Tr E(p).
\end{eqnarray}
\noindent
Adding up from (\ref{E:4.13}) to (\ref{E:4.16}) yields the following result, which is the second main result of this paper.
\begin{theorem}   \label{Theorem:4.1}
When ${\mathcal N}$ is a $2$-dimensional Riemannian manifold, the densities computing $v_{j}(\lambda)$ ($j = 0, 1, 2$) are given as follows.
\begin{eqnarray*}
v_{0}(\lambda)(p) & = & \frac{r_{0}}{8 \pi} ,  \qquad v_{1}(\lambda)(p) ~ = ~ 0, \\
v_{2}(\lambda)(p) & = &
r_{0} \left\{ - \frac{\lambda}{4 \pi} ~ + ~ \frac{\tau_{M}(p)}{32 \pi} ~ + ~ \frac{\tau_{{\mathcal N}}(p)}{96 \pi} ~ - ~ \frac{H_{1}^{2}(p)}{32 \pi}
~ + ~ \frac{H_{2}(p)}{32 \pi} \right\} + \frac{1}{4 \pi} \Tr E(p).
\end{eqnarray*}
Integrating these densities on ${\mathcal N}$, for $t \rightarrow 0^{+}$ we get the following result.
\begin{eqnarray*}
\Tr e^{-t R(\lambda)} & \sim & \frac{r_{0}}{8 \pi} \vol({\mathcal N}) t^{-2} + \int_{{\mathcal N}} v_{2}(\lambda)(p) d \vol({\mathcal N}) + O(t).
\end{eqnarray*}
\end{theorem}
\noindent
Lemma \ref{Lemma:3.2} together with (\ref{E:3.14}) leads to the following result, which is obtained in \cite{KL3} when a collar neighborhood
of ${\mathcal N}$ is a warped product manifold.
\begin{corollary}
\begin{eqnarray*}
\zeta_{R(\lambda)}(0) & = & \int_{{\mathcal N}} v_{2}(\lambda)(p) ~ d \vol({\mathcal N})  \\
& = & 2 \int_{{\mathcal N}}
\left\{ r_{0} \left( - \frac{\lambda}{8 \pi} ~ + ~ \frac{\tau_{M}(p)}{64 \pi} ~ + ~ \frac{\tau_{{\mathcal N}}(p)}{192 \pi} ~ - ~ \frac{H_{1}^{2}(p)}{64 \pi} ~ + ~ \frac{H_{2}(p)}{64 \pi} \right) + \frac{1}{8 \pi} \Tr E(p) \right\} d \vol({\mathcal N})   \\
& = & 2 \left( \zeta_{(\Delta_{M} + \lambda)}(0) - \zeta_{(\Delta_{M_{0}, D} + \lambda)}(0) \right).
\end{eqnarray*}
\end{corollary}
\noindent
{\it Remark} : It is interesting to compare Theorem \ref{Theorem:4.1} with the results of \cite{PS} and  \cite{Li}.
When $(X, Y, g)$ is a compact Riemannian manifold with boundary $Y$ and ${\mathcal E}$ is the trivial line bundle,
we consider the Steklov problem as in \cite{Li} and \cite{PS}.
Then, there appears the Dirichlet-to-Neumann operator ${\mathcal D}_{DN} : C^{\infty}(Y) \rightarrow C^{\infty}(Y)$.
Roughly, this operator is a half of the Dirichlet-to-Neumann operator $R(0)$ discussed in this paper.
When $\Dim X = 3$ (and so $\Dim Y = 2$), it is known that $\Tr e^{-t {\mathcal D}_{DN}} ~ \sim ~ A_{0} t^{-2} + A_{1} t^{-1} + A_{2} + O(t)$,
where $A_{j} = \int_{Y} A_{j}(y) d\vol(Y)$, $0 \leq j \leq 2$.
It is known \cite{Li,PS} that
\begin{eqnarray*}
A_{0}(y) = \frac{1}{2 \pi}, \qquad  A_{1}(y) = \frac{H_{1}(y)}{4 \pi}, \qquad
A_{2}(y) = \frac{\tau_{X}(y)}{32 \pi} + \frac{\tau_{Y}(y)}{96 \pi} + \frac{H_{1}^{2}(y)}{16 \pi}.
\end{eqnarray*}
%\bibliographystyle{plain}
%\bibliography{addarchive,addbook,archive,addunpub}

\begin{thebibliography}{BFK}

\bibitem{ALM} L. Alias, J. de Lira and J.M. Malacarne, \emph{Constant higher-order mean curvature hypersurfaces in Riemannian spaces},
J. of Inst. Math. Jussieu. \textbf{5(4)}  (2006), 527-562.


\bibitem{BFK} D. Burghelea, L. Friedlander and T. Kappeler, \emph{Mayer-Vietoris type formula for determinants of elliptic differential operators},
J. of Funct. Anal. \textbf{107}  (1992), 34-66.

\bibitem{BGV} N. Berlin, E. Getzler and M. Vergne, \emph{Heat Kernels and Dirac Operators},
Springer-Verlag,  1992.


\bibitem{Ca}
 G. Carron, \emph{D\'eterminant relatif et fonction Xi}, Amer. J. Math. \textbf{124}  (2002), 307-352.

\bibitem{eliz95b}
E.~Elizalde,
\emph{Ten Physical Applications of Spectral Zeta Functions},
Lect. Notes Phys. (Springer Verlag, Berlin, 2012) 855:1--225.

\bibitem{Fo}
 R. Forman, \emph{Functional determinants and geometry}, Invent. Math.  \textbf{88}  (1987), 447-493.


\bibitem{FK} G. Fucci and K. Kirsten, \emph{The spectral zeta function for Laplace operators on warped product manifolds of type $I \times_{f} N$}, Commun. Math. Phys. \textbf{317}  (2013), 635-665.


\bibitem{Gi1} P. B. Gilkey, \emph{Invariance Theory, the Heat Equation, and the Atiyah-Singer Index Theorem},
2nd Edition, CRC Press, Inc., 1994.


\bibitem{Gi2} P. B. Gilkey,
\emph{Asymptotic Formulae in Spectral Geometry} Chapman and Hall/CRC, 2003.


\bibitem{GS} G. Grubb and R. Seeley,
\emph{Weakly parametric pseudodifferential operators and Atiyah-Patodi-Singer boundary problems},
Invent. Math.  \textbf{121}  (1995), 481-529


\bibitem{Ki1} K. Kirsten, \emph{Spectral Functions in Mathematics and Physics}, Chapman and Hall/CRC, 2002


\bibitem{KL1} K. Kirsten and Y. Lee,
\emph{The Burghelea-Friedlander-Kappeler-gluing formula for zeta-determinants on a warped product manifold and a product manifold},
J. Math. Phys.  \textbf{58}  (2015), no. 12, 123501-1-19.


\bibitem{KL2} K. Kirsten and Y. Lee, \emph{The BFK-gluing formula and relative determinants on manifolds with cusps},
J. Geom. Phys.  \textbf{117}  (2017), 197-213.


\bibitem{KL3} K. Kirsten and Y. Lee,
\emph{The polynomial associated with the BFK-gluing formula of the zeta-determinant on a compact warped product manifold}, to appear in J. Geom. Anal.

\bibitem{Le1}
 Y. Lee, \emph{Mayer-Vietoris formula for the determinant of a Laplace operator on an even-dimensional manifold},
Proc. Amer. Math. Soc. \textbf{123}, no. \textbf{6},  (1995), 1933-1940.


\bibitem{Le2}
 Y. Lee, \emph{Burghelea-Friedlander-Kappeler's gluing formula for the zeta determinant and its applications
to the adiabatic decompositions of the zeta-determinant and the analytic torsion}, Trans. Amer. Math. Soc.
\textbf{355}, no. \textbf{10},  (2003), 4093-4110.


\bibitem{LU} J. Lee and G. Uhlmann, \emph{Determining isotropic real-analytic conductivities by boundary measurements},
Comm. Pure Appl. Math.  \textbf{42}  (1989), 1097-1112.


\bibitem{Li}
 G. Liu, \emph{Asymptotic expansion of the trace of heat kernel associated to the Dirichlet-to-Neumann opetrator}, J. Diff. Equations.
\textbf{259}, (2015), 2499-2545.


\bibitem{MM}
J. M\"uller and W. M\"uller, \emph{Regularized determinants of Laplace type operators, analytic surgery and relative determinants}, Duke. Math. J.
\textbf{133}  (2006), 259-312.

\bibitem{PW}
J. Park and K. Wojciechowski, \emph{Adiabatic decomposition of the zeta-determinant and Dirichlet-to-Neumann operator},
J. Geom. Phys. \textbf{55}, (2005), 241-266.


\bibitem{PS}
I. Polterovich and D. A. Sher, \emph{Heat invariants of the Stekelov problem}, J. Geom. Anal. \textbf{25}  (2015), 924-950.

\bibitem{ray71-7-145}
D. B. Ray and I. M. Singer, \emph{R-torsion and the Laplacian on Riemannian manifolds}, Adv. Math. \textbf{7} (1971) 145-210.


\bibitem{Sh}
M. A. Shubin, \emph{Pseudodifferential operators and spectral theory}, Springer-Verlag, 1987.


\bibitem{Vi}
J. Viaclovsky, \emph{Topics in Riemannian geometry Lecture notes}, Available online at http://www.math.wisc.edu/~jeffv/courses/$865_Fall_2011$.pdf

\bibitem{Vo}
W. Voros, \emph{Spectral function, special functions and Selberg zeta function},
Commun. Math. Phys. \textbf{110}  (1987), 439-465.


\end{thebibliography}
%%%%%%%%%%%%%%%%%%%%%%%%%%%%%%%%%%%%%%%%

\end{document}